\documentclass[reqno]{amsart}

\usepackage{enumitem}
\usepackage{nicefrac}
\usepackage[perpage,symbol*]{footmisc}
\usepackage{bbm}
\usepackage{hyperref}
\usepackage{amsmath,amssymb}

\DefineFNsymbols{daggers}{\dag\ddag{\dag\dag}{\ddag\ddag}}
\setfnsymbol{daggers}

\newcommand*{\mailto}[1]{\href{mailto:#1}{\nolinkurl{#1}}}
\newcommand{\arxiv}[1]{\href{http://arxiv.org/abs/#1}{arXiv:#1}}

\newtheorem{theorem}{Theorem}[section]
\newtheorem{definition}[theorem]{Definition}
\newtheorem{lemma}[theorem]{Lemma}
\newtheorem{proposition}[theorem]{Proposition}
\newtheorem{corollary}[theorem]{Corollary}

\newtheorem{remark}[theorem]{Remark}

\newcommand{\R}{{\mathbb R}}
\newcommand{\N}{{\mathbb N}}

\newcommand{\C}{{\mathbb C}}

\newcommand{\Wr}{\mathsf{w}}

\newcommand{\Sr}{\mathrm{s}}
\newcommand{\E}{\mathrm{e}}
\newcommand{\I}{\mathrm{i}}

\newcommand{\supp}{\mathrm{supp}}

\newcommand{\re}{\mathrm{Re}}
\newcommand{\loc}{{\mathrm{loc}}}
\newcommand{\cc}{{\mathrm{c}}}

\newcommand{\T}{\mathrm{T}}

\newcommand{\wt}{\tilde}

\newcommand{\OO}{\mathcal{O}}
\newcommand{\oo}{o}

\newcommand{\ledot}{\,\cdot\,}
\newcommand{\redot}{\cdot\,}

\newcommand{\Iso}[1]{\mathrm{Iso}(#1)}

\newcommand{\qd}{{[1]}}

\newcommand{\dip}{\upsilon}

\newcommand{\SM}{\mathcal{R}}
\newcommand{\Peakons}{\mathcal{P}}

\newcommand{\CHdom}{\mathcal{D}}
\newcommand{\String}{\mathcal{S}}

\makeatletter
\@namedef{subjclassname@2020}{%
  \textup{2020} Mathematics Subject Classification}
  
\newcommand{\dlmf}[1]{%
\cite[%
 \def\nextitem{\def\nextitem{, }}%
 \@for \el:=#1\do{\nextitem\expandafter\dlmf@eq@href\el...\end}%
]{dlmf}%
}
\def\dlmf@eq@href#1.#2.#3.#4\end{%
  \href{http://dlmf.nist.gov/#1.#2.E#3}{(#1.#2.#3)}}
\makeatother  

\numberwithin{equation}{section}

\begin{document}

\title[Camassa--Holm flow with step-like initial data]{The conservative Camassa--Holm flow\\ with step-like irregular initial data}
 
\author[J.\ Eckhardt]{Jonathan Eckhardt}
\address{Department of Mathematical Sciences\\ Loughborough University\\ Epinal Way\\ Loughborough\\ Leicestershire LE11 3TU \\ UK}
\email{\mailto{J.Eckhardt@lboro.ac.uk}}

\author[A.\ Kostenko]{Aleksey Kostenko}
\address{Faculty of Mathematics and Physics\\ University of Ljubljana\\ Jadranska ul.\ 19\\ 1000 Ljubljana\\ Slovenia\\ and 
Faculty of Mathematics\\ University of Vienna\\ Oskar-Morgenstern-Platz 1\\ 1090 Vienna\\ Austria}
\curraddr{Faculty of Mathematics and Physics\\ University of Ljubljana\\ Jadranska 19\\ 1000 Ljubljana\\ Slovenia\\ and Institute for Analysis and Scientific Computing\\Vienna University of Technology\\Wiedner Hauptstra\ss e 8–10/101\\
1040 Wien\\ Austria}\email{\mailto{Aleksey.Kostenko@fmf.uni-lj.si}}

\thanks{\href{http://dx.doi.org/10.1112/plms.70050}{Proc.\ Lond.\ Math.\ Soc.\ (3) {\bf 130} (2025), no.~5, Paper No.~e70050, 42pp.}}

\thanks{{\it Research of A.K.\ is supported by the Austrian Science Fund (FWF) under Grant I-4600 and by the Slovenian Research Agency (ARRS) under Grant No.\ N1-0137.}}


\begin{abstract}
 We extend the inverse spectral transform for the conservative Camassa--Holm flow on the line to a class of initial data that requires strong decay at one endpoint but only mild boundedness-type conditions at the other endpoint. 
 The latter condition appears to be close to optimal in a certain sense for the well-posedness of  the conservative Camassa--Holm flow.
 As a byproduct of our approach, we also find a family of new (almost) conservation laws for the Camassa--Holm equation, which could not be deduced from its bi-Hamiltonian structure before and which are connected to certain Besov-type norms (however, in a rather involved way). 
 These results appear to be new even under positivity assumptions on the corresponding momentum, in which case the conservative Camassa--Holm flow coincides with the classical Camassa--Holm flow and no blow-ups occur.
\end{abstract}

\dedicatory{Dedicated to the memory of Daphne Jane Gilbert}

\maketitle

\section{Introduction}

The Camassa--Holm equation 
 \begin{align}\label{eqnCH}
    u_{t} -u_{xxt}  = 2u_x u_{xx} - 3uu_x + u u_{xxx}
 \end{align}
was first noticed by Fokas and Fuchsteiner~\cite{fofu} to be formally integrable. 
Its weak formulation as a non-local conservation law
\begin{align}
\label{eqnCHweak}
  u_t + u u_x + P_x  = 0, 
\end{align}
 where the source term $P$ is defined (for suitable functions $u$) as the convolution 
 \begin{align}\label{eq:PequCHdef}
 P = \frac{1}{2}\E^{-|\cdot|} \ast \biggl(u^2 + \frac{1}{2}u_x^2\biggr), 
 \end{align} 
is reminiscent of the three dimensional incompressible Euler equation. 
  It turned out that equation~\eqref{eqnCH} exhibits a rich mathematical structure and its intensive study started with the discoveries of Camassa and Holm~\cite{caho93} (we only refer to a very brief selection of articles~\cite{besasz00, bokoshte09, brco07, co01, co05, coes98, como00, cost00, geho08, hora07, le05b, mc03, mc04, mis, mis02, xizh00}).  
  For instance, the Bott--Virasoro group serves for equation~\eqref{eqnCHweak} as a symmetry group and thus the Camassa--Holm equation can be viewed as a geodesic equation with respect to a right invariant metric~\cite{mis}.   
   Among many celebrated models for shallow water waves (for example, the Korteweg--de Vries equation or the Benjamin--Bona--Mahoney equation), the significance of equation~\eqref{eqnCH} stems from the fact that it is the first (and long sought after~\cite{with}) model featuring the following properties: 
   (a)~complete integrability, as it admits a Lax pair formulation, 
   (b)~solitons, including peaked ones (so-called {\em peakons}), and 
   (c)~finite time blow-up of smooth solutions that resembles wave-breaking to some extent (see~\cite{coes98b}). 
  More specifically, singularities develop in a way that the solution $u$ remains bounded pointwise, while the spatial derivative $u_x$ tends to $-\infty$ at some points. 
  However, the $H^1$ norm of $u$ remains bounded (actually, the $H^1$ norm serves as a Hamiltonian for the Camassa--Holm equation and is thus conserved) and $u$ approaches a limit weakly in $H^1$ as the blow-up happens, which raises the natural question of continuation past the blow-up. 
  Indeed, it was found that the Camassa--Holm equation possesses global weak solutions~\cite{xizh00} which are not unique however and hence continuation of solutions after blow-up is a delicate matter. 
 
  A particular kind of global weak solutions are the so-called {\em conservative solutions}, the notion of which was suggested independently in~\cite{brco07} and~\cite{hora07}.   
  These are weak solutions $(u,\mu)$ of the system 
  \begin{align}
 \begin{split}\label{eqnOurCH}
  u_t + u u_x + P_x & = 0, \\
  \mu_t + (u\mu)_x & = (u^3 - 2Pu)_x, 
 \end{split}
 \end{align}
 where the auxiliary function $P$ satisfies
 \begin{align}\label{eq:PequDef}
  P - P_{xx} & = \frac{u^2+ \mu}{2}.
 \end{align} 
We will call system~\eqref{eqnOurCH} the {\em two-component Camassa--Holm system} because it includes the Camassa--Holm equation as well as its two-component generalization (see \cite{chlizh06, coiv08, esleyi07, hoiv11})
 \begin{align}
 \begin{split}\label{eqn2CH}
    u_{t} -u_{xxt} & = 2u_x u_{xx} - 3uu_x + u u_{xxx} - \varrho \varrho_x, \\
   \varrho_t & = -u_x \varrho - u\varrho_x,
 \end{split}
 \end{align}
where one just needs to set $\mu = u^2 + u_x^2 + \varrho^2$. 
The role of the additional positive Borel measure $\mu$ is to control the loss of energy at times of blow-up. 
More precisely, as the model example of a peakon--antipeakon collision illustrates (see~\cite[\S~6]{brco07} for example), at times when the solution blows up, the corresponding energy, measured by $\mu$, concentrates on sets of Lebesgue measure zero.
Solutions of this kind have been constructed by a generalized method of characteristics that relied on a transformation from Eulerian to Lagrangian coordinates and was accomplished for various classes of initial data in~\cite{brco07, hora07, grhora12, grhora12b}.
 
From our point of view, conservative solutions are of special interest because they preserve the integrable structure of the Camassa--Holm equation and can be obtained by employing the inverse spectral transform method~\cite{ConservMP,ConservCH}. 
This approach is based on the solution of an inverse problem, which is equivalent to that for so-called {\em generalized indefinite strings}~\cite{IndefiniteString}, a spectral problem that generalizes Krein strings and originated from work of Krein and Langer~\cite{krla79} on the indefinite moment problem and having close connections with $2\times 2$ canonical systems, a central object of Krein--de Branges theory.
The main purpose of this article is to establish existence and stability of global weak solutions for the two-component Camassa--Holm system~\eqref{eqnOurCH} with initial data of low regularity and relatively mild boundedness restrictions.    
More specifically, we will consider this system in the phase space $\CHdom$ defined below. 

\begin{definition}\label{def:Dspace}
 The set $\CHdom$ consists of all pairs $(u,\mu)$ such that $u$ is a real-valued function in $H^1_{\loc}(\R)$ and $\mu$ is a positive Borel measure on $\R$ with
\begin{align}\label{eqnmuac}
   \int_B u(x)^2 + u'(x)^2\, dx \leq  \mu(B)
\end{align}
for every Borel set $B\subseteq\R$, satisfying the asymptotic growth restrictions  
\begin{align}
 \label{eqnMdef-}   \int_{-\infty}^0 \E^{-s} \bigl(u(s)^2 + u'(s)^2 \bigr) ds +  \int_{-\infty}^{0} \E^{-s} d\dip(s) & < \infty,  \\
 \label{eqnMdef+}   \limsup_{x\rightarrow\infty}\, \E^{x} \biggl(\int_{x}^{\infty}\E^{-s}(u(s) + u'(s))^2ds + \int_{x}^{\infty}\E^{-s}d\dip(s)\biggr) & < \infty,
\end{align}
where $\dip$ is the positive Borel measure on $\R$ defined such that 
\begin{align}\label{eqndipdef}
 \mu(B) = \dip(B) + \int_B u(x)^2 + u'(x)^2\, dx. 
\end{align}
\end{definition}

We prefer to work with pairs $(u,\mu)$ and the unusual condition~\eqref{eqnmuac} instead of the simpler definable pairs $(u,\dip)$ since, for example, the measure $\mu$ is more natural when considering suitable notions of convergence; see Definition~\ref{def:topolD} and Section~\ref{secSpace} for details.
Condition~\eqref{eqnMdef-} in this definition requires strong decay of both, the function $u$ and the measure $\dip$, at $-\infty$ (in particular, the condition on $u$ in~\eqref{eqnMdef-} is equivalent to the function $x\mapsto\E^{\nicefrac{-x}{2}}u(x)$ belonging to $H^1$ near $-\infty$), whereas condition~\eqref{eqnMdef+}\footnote{Upon making a simple change of variables and denoting $h(x) = (u+u')^2\circ \log(1/x)$, condition~\eqref{eqnMdef+} is obviously related to the boundedness of the classical Hardy operator $h\mapsto \frac{1}{x}\int_0^x h(s)ds$ (see~\cite{edev, muc} for example).} 
  on the growth near $+\infty$ is rather mild and satisfied as soon as $u+u'$ is bounded and $\dip$ is a finite measure for example.
In particular, these conditions allow initial data that is strongly decaying near $-\infty$ and eventually periodic/almost periodic near $+\infty$ or step-like profiles that are asymptotically constant near $+\infty$.
Of course, the conditions at $-\infty$ and $+\infty$ could also be switched due to the symmetry
\begin{align}
  (x,t) \mapsto (-x,-t)
\end{align}
of the two-component Camassa--Holm system. 
However, we are not able to allow nonzero asymptotics at both endpoints because the corresponding inverse spectral and scattering theory are not sufficiently well developed.\footnote{Let us mention that despite some similarities with the recent work of Grudsky and Rybkin~\cite{grry15,ry18} on the Korteweg--de Vries equation with step-like initial data, our results seem to be essentially different. On the other hand, the simplicity of the spectrum of the underlying Lax operator may indicate some parallels with the study of the Benjamin--Ono equation on the line~\cite{wu16,wu17,sau} (especially when one additionally assumes that $u-1\in H^1$ near $+\infty$) or with the cubic Szeg\H{o} equation on the torus~\cite{GG17,  gpt}. However, in contrast to the Benjamin--Ono, the latter case admits a complete solution of the corresponding inverse spectral problem, see the forthcoming~\cite{ISPforCH}.}  

Existence and stability of global weak solutions for initial data in the phase space $\CHdom$ will be established by means of the inverse spectral transform method (originally developed for the Korteweg--de Vries equation~\cite{abcl91, abse81, ecvH81} by Gardner, Greene, Kruskal and Miura~\cite{gagrkrmi67}). 
In fact, our results imply well-posedness of the conservative Camassa--Holm flow on $\CHdom$ in some sense (see Remark~\ref{rem:BressanUniq} for further details). 
To this end, we will consider the associated isospectral problem  
\begin{align}\label{eqnISP}
 -f'' + \frac{1}{4} f = z\, \omega f + z^2 \dip f, 
\end{align}
where $z$ is a complex spectral parameter and $\omega = u - u_{xx}$ in a distributional sense. 
Despite a large amount of articles, relatively little is known about inverse problems for~\eqref{eqnISP} when $\omega$ is allowed to be indefinite and $\dip$ to be not zero. 
In fact, most of the literature on this subject restricts to the case when $\dip$ vanishes identically and $\omega$ is strictly positive and smooth (in which case the spectral problem can be transformed into a standard potential form that is known from the Korteweg--de Vries equation~\cite{besasz98, le04, mc03b}).  
Apart from the explicitly solvable finite dimensional case~\cite{besasz00, ConservMP, InvPeriodMP} and~\cite{ConservCH}, only insufficient partial uniqueness results~\cite{be04, bebrwe08, bebrwe12, bebrwe15, LeftDefiniteSL, CHPencil, TFCPeriod, IsospecCH} have been obtained so far for the inverse problem in the indefinite case.  

Our results rely on recent progress in the spectral theory of generalized indefinite strings and canonical systems achieved in, respectively,~\cite{IndefiniteString, DSpec} and~\cite{rowo20}. 
In particular, the results obtained there imply that due to condition~\eqref{eqnMdef-}, the underlying spectrum will turn out to be simple, whereas condition~\eqref{eqnMdef+} guarantees that zero does not belong to the spectrum. 
 As a consequence, the associated spectral data consists of a spectral measure $\rho$ without mass near zero. 
Motivated by the case of conservative multi-peakons~\cite{ConservMP} and the well-known time evolution of spectral data for classical solutions of the Camassa--Holm equation~\cite[Section~6]{besasz98}, it is natural to define a flow on these associated spectral measures by setting  
\begin{align}\label{eqnCHevolint}
  \rho(B,t) = \int_B \E^{-\frac{t}{2\lambda}} d\rho(\lambda,0)
\end{align} 
for every Borel set $B\subseteq\R$. 
In order for this to be well-defined, it is desirable for the spectral measures $\rho$ not to have mass around zero, which is why we impose the growth restriction~\eqref{eqnMdef+}.
In fact, it will turn out that condition~\eqref{eqnMdef+} is necessary and sufficient for a spectral gap around zero to exist. 
From this perspective, the restriction~\eqref{eqnMdef+} appears to be nearly optimal; see Remark~\ref{rem:Optimal} for further details. 
The flow defined in this way gives rise to global weak solutions of the two-component Camassa--Holm system by means of the inverse spectral transform. 
We will keep referring to these solutions as {\em conservative} even though it is not obvious what quantities are actually conserved (as long as it is not additionally assumed that the measure $\mu$ is finite, in which case we recover the global conservative solutions constructed in~\cite{brco07} and~\cite{hora07} by a different approach). 
On the other hand, in view of the spectral data, this is clear however: The spectrum as well as the spectral types (the absolutely continuous spectrum, the singular continuous spectrum and the point spectrum) are preserved over time since spectral measures for different times are mutually absolutely continuous. 
As a consequence, any property of $(u,\mu)$ that can be read off the spectrum is conserved too.
We will use this observation to obtain  a continuous family of almost conserved quantities (that is, quantities that allow a global in time upper and lower bound) that appear to be new and do not follow from the bi-Hamiltonian structure directly. 
Let us also mention here that we will see in Section~\ref{secIST} that indeed any kind of simple spectrum can occur as long as there is a gap in the spectrum around zero. 

This approach to obtain global weak solutions of the two-component Camassa--Holm system~\eqref{eqnOurCH} can also be seen as an extension by continuation of the conservative Camassa--Holm flow from multi-peakons to the whole phase space $\CHdom$.\footnote{It was pointed out to us by one of the referees that this approach is reminiscent of the construction of solutions to the Korteweg--de Vries equation by considering a closure of multi-soliton solutions in suitable topologies (see, for instance, \cite{gkz92,hmm21,kot24,lun90,mar91}). One may also think of this in the vein of so-called soliton gases (see~\cite{elka,zak}).} 
More precisely, it was shown in~\cite{ConservMP} that the conservative Camassa--Holm flow is completely integrable in the case of multi-peakons, that is, that~\eqref{eqnISP} is a corresponding isospectral problem and that the evolution of spectral data is governed by~\eqref{eqnCHevolint}. 
It was shown later in~\cite{ConservCH} that the analysis extends to the case when the initial data has sufficiently fast decay at both endpoints (when one imposes the same decay condition as in~\eqref{eqnMdef-} also at $+\infty$). 
However, recent results in~\cite{DSpec} enable us to extend the conservative Camassa--Holm flow to a much wider class of initial data, namely, to all of $\CHdom$.
 To the best of our knowledge, the whole phase space $\CHdom$ is not covered by any previous constructions of weak solutions. 
 Moreover, it will turn out that the novel quantity in condition~\eqref{eqnMdef+} serves as an almost conservation law in the sense that the quantity 
\begin{align}
E(u,\mu) = \sup_{x\in \R}\, \E^{\frac{x}{2}} \biggl(\int_{x}^{\infty}\E^{-s}(u(s) + u'(s))^2ds + \int_{x}^{\infty}\E^{-s}d\dip(s)\biggr)^{\nicefrac{1}{2}}
\end{align}
admits global in time bounds. 
More precisely, if the pair $(u,\mu)$ is a conservative solution to the two-component Camassa--Holm system~\eqref{eqnOurCH}, then
\begin{align}\label{eq:mainConservLaw}
\frac{1}{6\sqrt{2}}E(u(\ledot,0),\mu(\ledot,0))\le E(u(\ledot,t),\mu(\ledot,t)) \le 6\sqrt{2}\, E(u(\ledot,0),\mu(\ledot,0)).
\end{align}
Indeed, the quantity $E(u,\mu)$ is controlled by the size of the spectral gap around zero and the latter is clearly preserved under the evolution defined by~\eqref{eqnCHevolint}. 
This observation will play a crucial role in our analysis. 

 Another advantage of our approach is that it completely linearizes the conservative Camassa--Holm flow. 
 More specifically, we will see in Section~\ref{secIST} that the phase space $\CHdom$ decomposes into invariant subsets, each of which can be parametrized by a subset of positive functions $\vartheta$ in $L^1_{\loc}(\R;\rho_0)$, where $\rho_0$ is some representative spectral measure. 
 With respect to this parametrization, the conservative Camassa--Holm flow becomes a simple linear flow given by
  \begin{align}
     \frac{\partial}{\partial t} \log\vartheta(\lambda,t) = -\frac{1}{2\lambda}.
  \end{align}
  This supports the view of the Camassa--Holm equation as a completely integrable Hamiltonian system and can be regarded as a kind of action-angle variables. 

\begin{remark}\label{rem:examples}
 As an example, let us mention a couple of pairs $(u,\mu)$ in $\mathcal{D}$ for which the corresponding spectral measure $\rho$ can be computed explicitly:
 \begin{enumerate}[label=(\roman*), ref=(\roman*), leftmargin=*, widest=ii]
    \item\label{itmISPforCH} 
Consider the pair $(u,\mu)$ in $\mathcal{D}$ with the function $u$ defined by 
  \begin{align}
    u(x) = 1-\E^{-x}\log(1+\E^x)
  \end{align}
  and such that the measure $\dip$ vanishes identically. 
  The corresponding spectral measure $\rho$ is then given by (use the transformation in Lemma~\ref{lemumuwdip} and~\cite[Example~I]{ISPforCH}) 
\begin{align}\label{eq:Rhoalpha}
    \rho(B) = \frac{1}{\pi} \int_{B\cap [1/4,\infty)}\sqrt{\lambda-\frac{1}{4}}\,\frac{d\lambda}{\lambda}.
    \end{align} 
    \item\label{itmLaguerre} More interesting examples stem from infinite multi-peakon profiles of the form
    \begin{align}
    u(x) & = \frac{1}{2}\sum_{n=1}^\infty p_n \E^{-|x-x_n|}.  
\end{align}
For instance, setting (we again suppose that the measure $\dip$ vanishes identically) 
\begin{align}
x_n & = \log\left( \sum_{k=1}^n L_{k-1}(-1)^2\right), & p_n & = \frac{\sum_{k=1}^n L_{k-1}(-1)^2}{nL_{n-1}(-1)L_{n}(-1)},
\end{align}
where the $L_n$ are the classical Laguerre polynomials~\cite{dlmf}, we end up with a shifted Laguerre weight as the corresponding spectral measure:
\begin{align}\label{eq:LaguerreMeasure}
    \rho(B) = \int_{B\cap [1,\infty)} \E^{1-\lambda} d\lambda.
\end{align}
It is not immediately clear that the function $u$ has a step-like form in this case, however,  in view of~\cite[Theorem~13.4]{ISPforCH}, it follows that $u - \frac{1}{4} \in H^1[0,\infty)$. 

On the other hand, one may take infinite multi-peakon profiles that are arranged periodically at infinity (for instance, take $x_n=n$ and $p_n = p_{n+N}$ for some $N\in\N$ and all $n\in\N$).
 In this case, one is again able to compute the corresponding spectral measure explicitly.   
  \end{enumerate}
Asymptotic long-time behavior of the global weak solutions corresponding to these initial profiles appears to be an intriguing topic, which will be addressed elsewhere. 
\end{remark}

Let us now briefly outline the content of this article. 
In Section~\ref{secSpace}, we discuss various properties of the phase space $\CHdom$, introduce a notion of convergence and relate it to some standard topologies. 
We will also establish a crucial connection between the set $\CHdom$ and a certain family of generalized indefinite strings, which will enable us to employ the direct and inverse spectral theory of generalized indefinite strings as developed earlier in~\cite{IndefiniteString,DSpec}. 

Section~\ref{secSP} deals with the direct spectral theory of~\eqref{eqnISP} with coefficients in $\CHdom$. 
Condition~\eqref{eqnMdef-} enables us to introduce a particular fundamental system of solutions to the differential equation~\eqref{eqnISP}, which is then used to introduce the key object of our spectral analysis; the Weyl--Titchmarsh function $m$ and the associated spectral measure $\rho$. 
It turns out that the Weyl--Titchmarsh function admits a particular integral representation, which guarantees that it is uniquely determined by the spectral measure in this case. 
The remaining part of this section includes results about spectral gap estimates as well as characterizations of coefficients that give rise to positive spectrum  and purely discrete spectrum.
 The importance of these results stems from the fact that all these spectral properties are conserved under the flow~\eqref{eqnCHevolint} and thus so are the corresponding asymptotic properties of solutions. 

We will then provide a complete solution to the inverse problem for~\eqref{eqnISP} with coefficients in $\CHdom$ in Section~\ref{secIST}. 
 First of all, we will show that the map $(u,\mu)\mapsto \rho$ is a bijection between the set $\CHdom$ and the set $\SM_0$ of all positive Poisson integrable Borel measures $\rho$ on $\R$ whose topological support has a gap around zero. 
 Moreover, we will establish continuity of this map with respect to suitable topologies and show that multi-peakon profiles are dense in $\CHdom$, which is necessary for being able to extend the conservative Camassa--Holm flow to all of $\CHdom$ later. 

In the final Section~\ref{secCCH}, we discuss the conservative Camassa--Holm flow on $\CHdom$. 
 More specifically, we define a flow on $\CHdom$ by means of the inverse spectral transform via~\eqref{eqnCHevolint}. 
 Our main result proves that this flow gives rise to global weak solutions to the two-component Camassa--Holm system~\eqref{eqnOurCH}. 
 The remaining part of this section provides a number of spatial asymptotic properties of solutions that are preserved by the flow as well as novel almost conserved quantities. 
 We note that these are connected with certain Besov-type norms, however, in a rather involved way (see Remark~\ref{rem:Besov}). 
 On the other hand, under additional positivity assumptions on the momentum $\omega= u-u_{xx}$, these conservation laws are simply given in terms of the $L^p$ norms of $u+u_x$ with $p\in (1/2,\infty]$.  
  
Let us finish this introduction with one more remark. 
Usually, when dealing with completely integrable partial differential equations, one first proves existence of weak or classical solutions in certain natural phase spaces (which are usually determined by corresponding conserved quantities).
 Afterwards, quite often under additional restrictive assumptions that are dictated by the need to solve direct and inverse spectral/scattering problems, one employs the inverse spectral/scattering transform approach to reconstruct and study these solutions.\footnote{The most illustrative example here is the Korteweg--de Vries equation, which is well-posed in $H^s$ for $s\ge -1$ (see~\cite{bou,kap06,kato83,kpv, kivi}). However, it is unclear at the moment how to apply the inverse spectral/scattering transform approach even in the case of initial data in  $H^1(\R)$.}  
 In this article, we are able to establish existence of global weak solutions by employing the inverse spectral transform approach for a large class of initial data for which we are unaware of previous existence results. 
 Moreover, the approach developed here suggests a rather large class of new (almost) conservation laws, which are hard to be noticed by the usual techniques using the bi-Hamiltonian structure. 
 Indeed, a serious obstacle in the case of the Camassa--Holm equation lies in the fact that the structure of the conserved quantities obtained from the bi-Hamiltonian structure is largely unknown (see~\cite{fisc99,le05} for example), which is in sharp contrast to the well-studied Korteweg--de Vries and Nonlinear Schr\"odinger equations.\footnote{However, even for these classical completely integrable models it was proved only recently that $H^s$-norms with negative $s\in [-1,0)$ for the KdV and $s\in (-1/2,0)$ for the NLS are (almost) conservation laws; see~\cite{bede,buko15,kota,kvz,ry10}.} 

\subsection*{Notation} 
 For an interval $I\subseteq\R$, we denote with $H^1_{\loc}(I)$, $H^1(I)$ and $H^1_{\cc}(I)$ the usual Sobolev spaces defined by 
\begin{align}
 H^1_{\loc}(I) & =  \lbrace f\in AC_{\loc}(I) \,|\, f'\in L^2_{\loc}(I) \rbrace, \\
 H^1(I) & = \lbrace f\in H^1_{\loc}(I) \,|\, f,\, f'\in L^2(I) \rbrace, \\ 
 H^1_{\cc}(I) & = \lbrace f\in H^1(I) \,|\, \supp(f) \text{ compact in } I \rbrace.
\end{align}
 The space of distributions $H^{-1}_{\loc}(I)$ is the topological dual space of $H^1_{\cc}(I)$. 
 A distribution in $H^{-1}_{\loc}(I)$ is said to be {\em real} if it is real for real-valued functions in $H^1_{\cc}(I)$. 

  For integrals of a function $f$ that is locally integrable with respect to a Borel measure $\mu$ on an interval $I$, we will employ the notation 
\begin{align}\label{eqnDefintmu}
 \int_x^y f\, d\mu = \begin{cases}
                                     \int_{[x,y)} f\, d\mu, & y>x, \\
                                     0,                                     & y=x, \\
                                     -\int_{[y,x)} f\, d\mu, & y< x, 
                                    \end{cases} 
\end{align}
 rendering the integral left-continuous as a function of $y$. 
 If the function $f$ is locally absolutely continuous on $I$ and $g$ denotes a left-continuous distribution function of the measure $\mu$, then we have the integration by parts formula 
\begin{align}\label{eqnPI}
  \int_{x}^y  f\, d\mu = \left. g f\right|_x^y - \int_{x}^y g(s) f'(s) ds,
\end{align}
 which will be used frequently throughout this article.

 \section{The phase space}\label{secSpace}

We are first going to elaborate on the nature of the {\em phase space} $\CHdom$, which we defined as the set of all pairs $(u,\mu)$ such that $u$ is a real-valued function in $H^1_{\loc}(\R)$ and $\mu$ is a positive Borel measure on $\R$ satisfying~\eqref{eqnmuac} and the asymptotic growth restrictions~\eqref{eqnMdef-} and~\eqref{eqnMdef+}, where the positive Borel measure $\dip$ is given by~\eqref{eqndipdef}.
With each pair $(u,\mu)$ in $\CHdom$ we shall associate another auxiliary pair $(\wt{\Wr},\wt{\dip})$, whose meaning and relevance will become clearer in the next section: 
 The function $\wt\Wr$ is defined almost everywhere on $(0,\infty)$ by  
 \begin{align}\label{eqnDefa}
 \wt{\Wr}(x) = -\frac{u(\log x) + u'(\log x)}{x}  
 \end{align}
 and the positive Borel measure $\wt{\dip}$ on $(0,\infty)$ is defined by 
 \begin{align}\label{eqnDefbeta}
  \wt{\dip}(B)   = \int_{\log(B)} \E^{- x} d\dip(x)=  \int_B \frac{1}{x}\, d\dip\circ\log(x)
 \end{align}
 for every Borel set $B\subseteq(0,\infty)$. 

\begin{lemma}\label{lem:propwt}
 If the pair $(u,\mu)$ belongs to $\CHdom$, then the corresponding function $\wt{\Wr}$ belongs to $L^2(0,\infty)$ and is real-valued almost everywhere, the corresponding measure $\wt{\dip}$ is finite and 
 \begin{align}\label{eqnMdefTilde}
  \limsup_{x\rightarrow\infty}\,  x \int_{x}^\infty  \wt{\Wr}(s)^2ds  + x \int_{x}^{\infty} d\wt{\dip} < \infty.
\end{align}
\end{lemma}

\begin{proof}
 We first note that $\wt\Wr$ is locally square integrable because $u$ belongs to $H^1_{\loc}(\R)$. 
 The growth restriction in~\eqref{eqnMdef-} implies that the function $\wt{\Wr}$ is square integrable near zero and that the measure $\wt{\dip}$ is finite near zero. 
 Moreover, upon performing a simple change of variables, condition~\eqref{eqnMdef+} is seen to imply~\eqref{eqnMdefTilde}, which also shows that $\wt\Wr$ is square integrable near $\infty$ and that $\wt\dip$ is finite near $\infty$. 
 Finally, since $u$ is real-valued, the function $\wt\Wr$ is real-valued almost everywhere.
\end{proof}

Let us define $\String_0$ as the set of all pairs $(\wt{\Wr},\wt{\dip})$ such that $\wt\Wr$ is an almost everywhere real-valued function in $L^2(0,\infty)$, $\wt\dip$ is a finite positive Borel measure on $[0,\infty)$ with no point mass at zero and the asymptotic condition~\eqref{eqnMdefTilde} holds. 
In view of Lemma~\ref{lem:propwt}, it follows that $(u,\mu)\mapsto(\wt\Wr,\wt\dip)$ defined by~\eqref{eqnDefa} and~\eqref{eqnDefbeta} maps the phase space $\CHdom$ to $\String_0$, where we extend the corresponding measure $\wt\dip$ defined in~\eqref{eqnDefbeta} to a finite positive Borel measure on $[0,\infty)$ by setting $\wt\dip(\{0\})=0$. 

\begin{lemma}\label{lemumuwdip}
  The mapping $(u,\mu)\mapsto(\wt\Wr,\wt\dip)$ is a bijection between $\CHdom$ and $\String_0$.
\end{lemma}

\begin{proof}
  For an arbitrary pair $(\wt\Wr,\wt\dip)$ in $\String_0$, the function $u$ defined by 
  \begin{align}\label{eqnuitofa}
   u(\log x) = - \frac{1}{x} \int_0^x  \wt{\Wr}(s) s\, ds 
  \end{align} 
  is real-valued, belongs to $H^1_{\loc}(\R)$ and satisfies 
  \begin{align*}
    \wt\Wr(x) =  -\frac{u(\log x) + u'(\log x)}{x}
  \end{align*} 
  for almost all $x>0$. 
  Furthermore, we define the positive Borel measure $\mu$ on $\R$ by~\eqref{eqndipdef}, where $\dip$ is the positive Borel measure on $\R$ such that~\eqref{eqnDefbeta} holds. 
  The pair $(u,\mu)$ then clearly satisfies~\eqref{eqnmuac}.
  Since the function $\wt\Wr$ is square integrable near zero and the measure $\wt\dip$ is finite near zero, it follows that
  \begin{align*}
    \int_{-\infty}^0 \E^{-s}(u(s)+u'(s))^2ds + \int_{-\infty}^0 \E^{-s}d\dip(s) < \infty.
  \end{align*}  
  From this we conclude that~\eqref{eqnMdef-} holds as well because 
     \begin{align*}
   \int_\R \E^{- s} u(s)^2 ds  = \int_0^\infty \frac{u(\log x)^2}{x^2}dx < \infty,
  \end{align*}  
  where finiteness of the last integral follows from boundedness of the Hardy operator on $L^2(0,\infty)$; see \cite[Theorem~319]{HLP}, \cite[Theorem~2.2.1]{edev}, \cite[Lemma~A.1]{SingMB} for example.
 As a change of variables translates condition~\eqref{eqnMdefTilde}  into property~\eqref{eqnMdef+}, we see that the pair $(u,\mu)$ indeed belongs to $\CHdom$.   
  It remains to note that the mapping $(\wt\Wr,\wt\dip)\mapsto(u,\mu)$ defined in this way is the inverse of the mapping in the claim.
 \end{proof}

\begin{lemma}\label{lem:umuEST}
 If the pair $(u,\mu)$ belongs to $\CHdom$, then the function $x\mapsto\E^{-\frac{x}{2}}u(x)$ belongs to $H^1(\R)$. 
 In particular, one has  
 \begin{align}\label{eq:uatinfty}
 \lim_{|x|\rightarrow\infty} \E^{-x}u(x)^2 = 0
 \end{align}
 and for all $x\in\R$ the bounds 
\begin{align}\label{eq:supU-EST}
 \frac{3}{2}\E^{-x}u(x)^2 \le \int_{-\infty}^x \E^{-s}d\mu(s)\le \int_\R \E^{-s}d\mu(s) < \infty.
  \end{align}
\end{lemma}
 
 \begin{proof}
Since the functions $\E^{-\frac{x}{2}}u(x)$ and $\E^{-\frac{x}{2}}(u(x)+u'(x))$ are square integrable for any pair $(u,\mu)$ in $\CHdom$, the first claim obviously follows from 
\[
  \frac{d}{dx}\E^{-\frac{x}{2}}u(x) = - \frac{1}{2}\E^{-\frac{x}{2}}u(x) + \E^{-\frac{x}{2}}u'(x) = - \frac{3}{2}\E^{-\frac{x}{2}}u(x) + \E^{-\frac{x}{2}}(u(x)+u'(x)). 
\]  
In order to prove the bounds in~\eqref{eq:supU-EST}, recall that $|f(x)| \leq \|f\|_{H^1(-\infty,x)}$ holds for all functions $f\in H^1(\R)$ and every $x\in\R$.
 From this we get after integrating by parts and using~\eqref{eq:uatinfty} that
  \begin{align*}
  \E^{-x}u(x)^2 & \le \int_{-\infty}^x \E^{-s}u(s)^2ds + \int_{-\infty}^x \E^{-s} \left(u'(s) - \frac{1}{2}u(s)\right)^2ds\\
  & =  \int_{-\infty}^x \E^{-s} \biggl(\frac{3}{4}u(s)^2 + u'(s)^2\biggr)ds - \frac{1}{2}\E^{-x}u(x)^2\\
  & \le \int_{-\infty}^x \E^{-s}d\mu(s) - \frac{1}{2}\E^{-x}u(x)^2.
  \end{align*}
  It remains to notice that the last integral in~\eqref{eq:supU-EST} is finite due to~\eqref{eqnMdef-}, \eqref{eqnMdef+} and
  \[
  \int_0^\infty \E^{-s}(u(s)^2+u'(s)^2)ds =  u(0)^2 - \int_0^\infty \E^{-s}u(s)^2ds + \int_0^\infty \E^{-s}(u(s)+u'(s))^2ds.  \qedhere
  \] 
 \end{proof}

\begin{remark}
Similar calculations as in the proof of Lemma~\ref{lem:umuEST}, using the $H^1$ norm on the semi-axis $(x,\infty)$ instead of $(-\infty,x)$, yield the analogous bound 
\begin{align}
 \E^{-x}u(x)^2 \le \int_x^{\infty} \E^{-s} \biggl(\frac{3}{2}u(s)^2 + 2u'(s)^2\biggr)ds.
  \end{align}
\end{remark}

For the sake of simplicity, instead of introducing a proper topology on the phase space $\CHdom$, we will only define a notion of convergence. 
It will follow from Proposition~\ref{propCont} in Section~\ref{secIST} that this mode of convergence is indeed induced by a metric. 

\begin{definition}\label{def:topolD}
We say that a sequence of pairs $(u_k,\mu_k)$ converges to $(u,\mu)$ in $\CHdom$ if  the functions $u_k$ converge to $u$ pointwise on $\R$ and  
\begin{align}\label{eqnContMU}
     \int_{-\infty}^x \E^{- s}d\mu_k(s)  \rightarrow \int_{-\infty}^x  \E^{- s}d\mu(s) 
\end{align}   
for almost every $x\in\R$.
 \end{definition}

Let us first compare this mode of convergence with more standard ones.

\begin{lemma}\label{lem:topD}
Let $(u_k,\mu_k)$ be a sequence of pairs that converge to $(u,\mu)$ in $\CHdom$. 
 \begin{enumerate}[label=(\roman*), ref=(\roman*), leftmargin=*, widest=iii]
    \item\label{itmtopDi} The functions $u_k$ converge to $u$ locally uniformly on $\R$. 
     \item\label{itmtopDii} For all $x\in\R$ one has 
\begin{align}\label{eqnContUint01}
     \int_{-\infty}^x u_k(s) ds & \rightarrow \int_{-\infty}^x u(s) ds, \\
   \label{eqnContUint01b}   \int_{-\infty}^x u_k(s)^2 ds & \rightarrow \int_{-\infty}^x u(s)^2 ds, \\
 \label{eqnContUint02}
 \int_{-\infty}^x \E^{-s}u_k(s)^2 ds  & \rightarrow \int_{-\infty}^x \E^{-s}u(s)^2 ds.
\end{align} 
 \item\label{itmtopDiii} The derivatives $u_k'$ converge to $u'$ weakly in $L^2_{\loc}(\R)$.
\item\label{itmtopDiv} For every continuous function $h\in C_{\cc}(\R)$ with compact support one has 
 \begin{align}
 \int_\R h(s)d\mu_k(s) \rightarrow \int_\R h(s)d\mu(s).
 \end{align}
 \end{enumerate}
\end{lemma}

\begin{proof}
 We first observe that the convergence in~\eqref{eqnContMU} entails that  
\begin{align*}  
  \sup_{k\in\N}\int_{-\infty}^x  \E^{- s}d\mu_k(s)  <\infty
  \end{align*}
  for all $x\in\R$. 
  In particular, this guarantees that the $H^1(I)$ norms of the functions $u_k$ are uniformly bounded for every compact interval $I$.
  Together with pointwise convergence of the functions $u_k$, this implies that the functions $u_k$ converge weakly in $H^1_{\loc}(\R)$, which proves~\ref{itmtopDi} and~\ref{itmtopDiii}.
  For a given $x\in\R$, the convergences in~\eqref{eqnContUint01} and~\eqref{eqnContUint01b} follow from dominated convergence because~\eqref{eq:supU-EST} in Lemma~\ref{lem:umuEST} gives the uniform bound 
\begin{align*}
u_k(s)^2 \le \E^s \sup_{l\in \N}\int_{-\infty}^x \E^{- r} d\mu_l(r)
\end{align*}
for all $s<x$ and $k\in\N$.
In order to prove the convergence in~\eqref{eqnContUint02}, we may use~\cite[Theorem~30.8]{ba01} because taking into account~\eqref{eqnContMU} one sees that for any $\varepsilon>0$ there is an $x_0 \in\R$ such that  
    \begin{align*}
   \int_{-\infty}^{x_0} \E^{-s} u_k(s)^2ds \leq \int_{-\infty}^{x_0} \E^{- s}d\mu_k(s) < \varepsilon
    \end{align*}
    for all $k\in\N$. 
   The remaining claim in~\ref{itmtopDiv} simply follows from pointwise convergence almost everywhere of the distribution function  in~\eqref{eqnContMU}.
\end{proof}

\begin{remark}
A few remarks are in order.
 \begin{enumerate}[label=(\roman*), ref=(\roman*), leftmargin=*, widest=iii]
    \item
Weak $L^2_{\loc}(\R)$ convergence in Lemma~\ref{lem:topD}~\ref{itmtopDiii} can not be replaced by strong $L^2_{\loc}(\R)$ convergence since this would imply that the functions $u_k$ converge to $u$ in $H^1_{\loc}(\R)$. 
 However, this is not the case as the crucial example of a peakon-antipeakon collision shows (see~\cite[Section~6]{brco07}).
\item Notice that when $x$ is taken to $\infty$ in~\eqref{eqnContMU} and~\eqref{eqnContUint02}, one only ends up with the inequality 
\begin{align}
\liminf_{k\to \infty}\int_{\R} \E^{- s}u_k(s)^2ds + \int_\R \E^{- s}d\mu_k(s) \ge \int_{\R} \E^{- s}u(s)^2ds + \int_{\R} \E^{- s}d\mu(s).
\end{align}
It turns out that under a mild additional assumption on the sequence $(u_k,\mu_k)$ it is possible to replace inequality by equality here (see Lemma~\ref{lem:convonR}). 
\end{enumerate}
\end{remark}

Using the bijection between $\CHdom$ and $\String_0$, we can also characterize the convergence in $\CHdom$ in terms of the corresponding pairs in $\String_0$ in a rather transparent way.

 \begin{lemma}\label{lem:topStr}
A sequence of pairs $(u_k,\mu_k)$ converges to $(u,\mu)$ in $\CHdom$ if and only if the corresponding pairs $(\wt\Wr_k,\wt\dip_k)$ and $(\wt\Wr,\wt\dip)$ in $\String_0$ satisfy 
  \begin{align}\label{eqnContAr} 
  \int_0^x \wt{\Wr}_{k}(s)ds  \rightarrow \int_0^x \wt{\Wr}(s)ds
    \end{align}
   for all $x>0$ and
   \begin{align} \label{eqnContvs} 
   \int_0^x  \wt{\Wr}_{k}(s)^2 ds + \int_0^x d\wt{\dip}_{k}  & \rightarrow \int_0^x \wt{\Wr}(s)^2 ds + \int_0^x d\wt{\dip}
  \end{align}
  for almost every $x>0$. 
  \end{lemma}  

\begin{proof}
Using the definition~\eqref{eqnDefa} of $\wt\Wr$, we get
\[
\int_0^x \wt{\Wr}(s)ds = -\int_0^x \frac{u(\log s) + u'(\log s)}{s}ds = -u(\log x) - \int_{-\infty}^{\log x}u(s)ds,
\]
which shows that~\eqref{eqnContAr} follows from pointwise convergence of $u_k$ together with~\eqref{eqnContUint01}. 
Furthermore, since one has 
\begin{align}\label{eq:changeofnorms}
\int_0^x \wt{\Wr}(s)^2 ds + \int_0^x d\wt{\dip} = \int_{-\infty}^{\log x}\E^{-s}d\mu(s) + \int_{-\infty}^{\log x}\E^{-s}u(s)^2ds + \frac{u(\log x)^2}{x},
\end{align}
 it follows from~\eqref{eqnContMU}, \eqref{eqnContUint02} and pointwise convergence of $u_k$ that~\eqref{eqnContvs} holds for almost all $x\in\R$.

In order to prove the converse, suppose that~\eqref{eqnContAr} holds for all $x>0$ and that~\eqref{eqnContvs} holds for almost all $x>0$, which clearly implies that 
\begin{align*}  
   \frac{3}{2} \sup_{k\in\N} \E^{-s} u_k(s)^2 \leq \sup_{k\in\N} \int_{-\infty}^{\log x}\E^{-r}d\mu_k(r) \leq \sup_{k\in\N}\int_0^x  \wt{\Wr}_{k}(r)^2 dr + \int_0^x d\wt{\dip}_{k}  <\infty
  \end{align*}
for all $x>0$ and $s\leq\log x$, where we also used~\eqref{eq:supU-EST} and~\eqref{eq:changeofnorms}. 
 Because of~\eqref{eqnuitofa}, we may infer from~\eqref{eqnContAr} that the functions $u_k$ converge to $u$ pointwise on $\R$. 
   Indeed, an integration by parts in~\eqref{eqnuitofa} gives 
\begin{align*}
-xu(\log x) = x\int_0^x \wt{\Wr}(s)\, ds - \int_0^x\int_0^s  \wt{\Wr}(r) dr\, ds
\end{align*}
for all $x>0$ and it remains to apply dominated convergence once again. 
Thus, we are left to show that~\eqref{eqnContMU} holds for almost all $x\in\R$. 
However, in view of~\eqref{eq:changeofnorms}, together with the already established pointwise convergence of the functions $u_k$, it suffices to prove that
\begin{align*}
\int_{-\infty}^{\log x}\E^{-s}u_k(s)^2ds \rightarrow \int_{-\infty}^{\log x}\E^{-s}u(s)^2ds
\end{align*}
for all $x>0$. 
But this again, as in the proof of Lemma~\ref{lem:topD}, can be deduced from~\cite[Theorem~30.8]{ba01} because for any $\varepsilon>0$ there is an $x_0>0$ such that 
\begin{align*}
  \int_{-\infty}^{\log x_0} \E^{-s}u_k(s)^2 ds \leq \int_{-\infty}^{\log x_0} \E^{-s}d\mu_k(s) \leq \int_0^{x_0} \wt\Wr_k(s)^2 ds + \int_0^{x_0} d\wt\dip_k < \varepsilon
\end{align*} 
for all $k\in\N$, which follows from the convergence in~\eqref{eqnContvs}.  
    \end{proof}
 
We conclude this section with a solution of the differential equation~\eqref{eq:PequDef}.

\begin{lemma}\label{lem:uniqP}
If the pair $(u,\mu)$ belongs to $\CHdom$, then the function $P$ defined on $\R$ by 
 \begin{align}\label{eq:P0def}
  P(x) =  \frac{1}{4} \int_\R \E^{-|x-s|} u(s)^2 ds +  \frac{1}{4} \int_\R \E^{-|x-s|} d\mu(s)
 \end{align}
 is the unique distributional solution to differential equation~\eqref{eq:PequDef} satisfying
 \begin{align}\label{eq:PasympInfty}
 P(x) = \oo\bigl(\E^{|x|}\bigr),\qquad |x|\rightarrow\infty.
 \end{align}
 \end{lemma}

\begin{proof}
 Since we may write 
 \begin{align}\label{eq:P0defSum}
 \begin{split}
  P(x) & = \frac{1}{4} \E^{-x}\int_{-\infty}^x \E^{s} u(s)^2 ds +  \frac{1}{4} \E^{-x}\int_{-\infty}^{x} \E^{s} d\mu(s)\\
  & \qquad\qquad + \frac{1}{4} \E^{x}\int_x^\infty \E^{-s} u(s)^2 ds +  \frac{1}{4} \E^{x}\int_{x}^{\infty} \E^{-s} d\mu(s),
  \end{split}
 \end{align}
  Lemma~\ref{lem:umuEST} guarantees that the function $P$ is well-defined.
 Moreover, it is straightforward to verify that $P$ is a distributional solution to~\eqref{eq:PequDef} and hence we only need to show that it satisfies the asymptotics~\eqref{eq:PasympInfty}. 
 First of all, one has 
 \begin{align*}
   P(x)  & \leq \frac{1}{4} \int_{\R} \E^{x-s} u(s)^2 ds +  \frac{1}{4} \int_{\R} \E^{x-s} d\mu(s)  = \OO(\E^{x}) = \oo(\E^{-x})
 \end{align*}
 as $x\rightarrow-\infty$. 
 On the other side, using~\eqref{eq:P0defSum} we get
  \begin{align*}
  P(x) =  \frac{1}{4} \E^{-x}\int_{-\infty}^x \E^{s} u(s)^2 ds +  \frac{1}{4}\E^{-x}\int_{-\infty}^{x} \E^{s} d\mu(s) + \oo(\E^{x})
 \end{align*}
 as $x\rightarrow\infty$ and it remains to apply dominated convergence to conclude that both of the two integrals are of order $\oo(\E^{2x})$ as $x\rightarrow\infty$. 
Finally, uniqueness follows from the well-known fact that distributional solutions to constant coefficient ordinary differential equations are classical solutions.
 \end{proof}
 
 Taking into account the uniqueness established in Lemma~\ref{lem:uniqP}, one can show that convergence in $\CHdom$ implies convergence of the corresponding solutions from Lemma~\ref{lem:uniqP} in the sense of distributions. 
 However, we will prove a considerably stronger convergence result in Corollary~\ref{cor:convPk} under a mild additional assumption.

 \section{The isospectral problem}\label{secSP}

In this section, we are going to introduce the spectral quantities that will linearize the conservative Camassa--Holm flow on $\CHdom$. 
For this purpose, let us fix an arbitrary pair $(u,\mu)$ in $\CHdom$ and define the distribution $\omega$ in $H^{-1}_{\loc}(\R)$ by
\begin{align}\label{eqnDefomega}
 \omega(h) = \int_\R u(x)h(x)dx + \int_\R u'(x)h'(x)dx, \quad h\in H_{\cc}^1(\R),
\end{align}
so that one has $\omega = u - u''$ in a distributional sense. 
We also recall that the positive Borel measure $\dip$ on $\R$ is defined by~\eqref{eqndipdef} and that it is always possible to uniquely recover the pair $(u,\mu)$ from the distribution $\omega$ and the measure $\dip$. 
 
Associated with the pair $(u,\mu)$ is the ordinary differential equation 
\begin{align}\label{eqnDE}
 - f'' + \frac{1}{4} f = z\, \omega f + z^2 \dip f, 
\end{align}
where $z$ is a complex spectral parameter. 
Due to the low regularity of the coefficients, this differential equation has to be understood in a distributional sense in general; see \cite{ConservCH, IndefiniteString, gewe14, sash03}:
  A solution of~\eqref{eqnDE} is a function $f\in H^1_{\loc}(\R)$ such that 
 \begin{align}\label{eqnDEweakform}
   \int_{\R} f'(x) h'(x) dx + \frac{1}{4} \int_\R f(x)h(x)dx = z\, \omega(fh) + z^2 \int_\R f h \,d\dip 
 \end{align} 
 for every function $h\in H^1_\cc(\R)$.
 We note that the derivative of such a solution $f$ is in general only defined almost everywhere. 
 However, there always is a unique left-continuous function $f^\qd$ on $\R$ such that 
 \begin{align}\label{eqnfqpm} 
     f^\qd = f' +\frac{1}{2} f - z (u +u') f 
\end{align} 
 almost everywhere on $\R$ (see \cite[Lemma~A.2]{ConservCH}), called the {\em quasi-derivative} of $f$. 

 The main consequence of the strong decay condition on the pair $(u,\mu)$ at $-\infty$ in~\eqref{eqnMdef-} is the existence of a particular fundamental system of solutions to the differential equation~\eqref{eqnDE} with certain prescribed asymptotics near $-\infty$.

\begin{theorem}\label{thmThetaPhi}
For every $z\in\C$ there are unique solutions $\phi(z,\redot)$ and $\theta(z,\redot)$ of the differential equation~\eqref{eqnDE} with the asymptotics 
\begin{align}\label{eqnphiasym}
  \phi(z,x) & \sim \E^{\frac{x}{2}}, & \theta(z,x) & \sim \E^{-\frac{x}{2}},   \\
\label{eqnthetaasym}
  \phi^\qd(z,x) & \sim \E^{\frac{x}{2}}, &  \theta^\qd(z,x) & = \oo\bigl(\E^{\frac{x}{2}}\bigr),   
\end{align}
as $x\rightarrow-\infty$. 
Moreover,  the derivative of the solution $\phi(z,\redot)$ is integrable and square integrable near $-\infty$. 
\end{theorem}

\begin{proof}
One can find a proof in~\cite[Theorem~2.1]{ConservCH}.
However, we provide the details since our further considerations will rely heavily on the corresponding constructions. 
First, notice that the claim is obvious when $z$ is zero since one can set 
 \begin{align*}
  \phi(0,x) & = \E^{\frac{x}{2}}, & \theta(0,x) & = \E^{-\frac{x}{2}},  
 \end{align*}
 so that we may assume that $z\not=0$.
 Let us suppose that two functions $f$ on $\R$ and $y$ on $(0,\infty)$ are related via
\begin{align}\label{eqnfg}
y(x) = \sqrt{x} f(\log x) = \frac{1}{\sqrt{\Sr'(x)}} f(\Sr(x)), 
\end{align} 
where the diffeomorphism $\Sr\colon (0,\infty)\rightarrow\R$ is simply given by  
 \begin{align*}
  \Sr(x) = \log x.
 \end{align*} 
   We first observe that $f$ belongs to $H^1_\loc(\R)$ if and only if  $y$ belongs to $H^1_{\loc}(0,\infty)$. 
   In this case, for a given function $h_\Sr$ in $H^1_\cc(0,\infty)$, a substitution yields 
   \begin{align*}
     \int_0^\infty y'(x)h_\Sr'(x)dx & = \int_0^\infty \biggl(f'(\Sr(x))+\frac{f(\Sr(x))}{2}\biggr) \biggl(h'(\Sr(x)) + \frac{h(\Sr(x))}{2}\biggr) \Sr'(x) dt \\
      & = \int_\R f'(x)h'(x)dx + \frac{1}{4} \int_\R f(x)h(x)dx,
   \end{align*} 
   where the functions $h\in H^1_\cc(\R)$ and $h_\Sr\in H^1_\cc(0,\infty)$ are related by 
   \begin{align*}
     h_\Sr(x)  = \sqrt{x} h(\log x).
    \end{align*}
     With the pair $(\wt\Wr,\wt\dip)$ defined by~\eqref{eqnDefa} and~\eqref{eqnDefbeta} as in Section~\ref{secSpace}, one computes that
   \begin{align}\label{eq:WrString}
     \int_0^\infty  \wt{\Wr}(x) (yh_\Sr)'(x)dx & = - \int_\R u(x)f(x) h(x)dx - \int_\R u'(x)(fh)'(x)dx,
   \end{align} 
   as well as
   \begin{align*}
     \int_{0}^{\infty} yh_\Sr\, d\wt{\dip} & = \int_{0}^\infty f(\Sr(x))h(\Sr(x)) d\dip\circ\Sr(x) =  \int_{\R} fh\, d\dip.
   \end{align*}
   With the help of these identities, we conclude that a function $f$ is a solution of the differential equation~\eqref{eqnDE} if and only if the function $y$ defined by \eqref{eqnfg} satisfies 
    \begin{align*}
   \int_{0}^\infty y'(x) h_\Sr'(x) dx  = z\, \wt{\omega}(yh_\Sr) + z^2 \int_{0}^{\infty} y h_\Sr \,d\wt{\dip} 
 \end{align*} 
 for every function $h_\Sr\in H^1_\cc(0,\infty)$, and thus $y$ is a (distributional) solution of the differential equation
       \begin{align}\label{eqntildeString}
    - y'' = z\,\wt{\omega}  y + z^2\wt{\dip} y,
   \end{align}
   where $\wt{\omega}$ is the real distribution in $H^{-1}_{\loc}[0,\infty)$ given by 
 \begin{align*}
 \wt{\omega}(g) = - \int_0^\infty \wt{\Wr}(x) g'(x)dx
 \end{align*}
 for all $g\in H^1_{\cc}[0,\infty)$, so that $\wt{\Wr}$ is the {\em normalized antiderivative} of $\wt\omega$.
  On the other hand (see~\cite[Section~6]{IndefiniteString}), solutions of~\eqref{eqntildeString} are closely related to the matrix solution $Y$ of the system
 \begin{align}\begin{split}\label{eqnCanSysString}
  Y(z,x) = \begin{pmatrix} 1 & 0 \\ 0 & 1 \end{pmatrix} & +  z\int_0^x \begin{pmatrix} - \wt{\Wr}(s) & 1 \\ - \wt{\Wr}(s)^2 &  \wt{\Wr}(s) \end{pmatrix} Y(z,s)ds \\ 
       & +  z\int_{0}^{x} \begin{pmatrix} 0 & 0 \\ -1 & 0 \end{pmatrix} Y(z,s) d\wt{\dip}(s), \quad x\in[0,\infty).
 \end{split}\end{align} 
Indeed, it is easy to see that each column of $Y(z,\redot)$ is of the form  
\begin{align*}
\begin{pmatrix} zy \\  y' + z \wt{\Wr} y \end{pmatrix},
\end{align*}
where $y$ is a solution to~\eqref{eqntildeString}. 
Since $\wt{\Wr}\in L^2_{\loc}[0,\infty)$ and $\wt{\dip}$ is a positive Borel measure on $[0,\infty)$,  for each $z\in\C$ there is a unique matrix solution $Y (z,\redot)$ on $[0,\infty)$ of the integral equation \eqref{eqnCanSysString}.
Taking into account that the quasi-derivative \eqref{eqnfqpm} of a solution $f$ given by \eqref{eqnfg} is expressed by 
\begin{align*}
\E^{-\frac{x}{2}}f^\qd(x) =  (y'+z\wt{\Wr}y)\circ \Sr^{-1}(x), 
\end{align*}
we can obtain the desired solutions $\phi(z,\redot)$ and $\theta(z,\redot)$ to \eqref{eqnDE} by setting
 \begin{align*}
 \begin{pmatrix}  \theta (z,x) & z\phi (z,x) \\ 
        z^{-1}\theta^\qd(z,x) & \phi ^\qd(z,x) \end{pmatrix} =  
       \begin{pmatrix} \E^{-\frac{x}{2}} & 0 \\ 0 & \E^{\frac{x}{2}} \end{pmatrix} Y (z,\Sr^{-1}(x)), \quad x\in\R.
 \end{align*}
Indeed, this immediately implies that $\theta(z,\redot)$,  $\theta^\qd(z,\redot)$ and $\phi^\qd(z,\redot)$ have the required asymptotics. 
 Regarding $\phi(z,\redot)$ however, we only get $\phi(z,x)= \oo(\E^{-\frac{x}{2}})$ as $x\rightarrow-\infty$. 
 Nevertheless, one can show that $Y_{12}(z,x)\sim zx$ as $x\to 0$ (see \cite[Equation~(2.7)]{ConservCH} for example), which yields the required asymptotic behavior of $\phi(z,\redot)$. 
 It remains to note that the asymptotic behavior in~\eqref{eqnphiasym} and~\eqref{eqnthetaasym} uniquely determines the solutions $\phi(z,\redot)$ and $\theta(z,\redot)$. 
\end{proof}

With these solutions, we are able to define the {\em Weyl--Titchmarsh function} 
\begin{align}\label{eq:m-funct-def}
  m(z) = -\lim_{x\rightarrow\infty} \frac{\theta(z,x)}{z\phi(z,x)}, \quad z\in\C\backslash\R,
\end{align}
which is a Herglotz--Nevanlinna function in view of~\cite[Section~5]{IndefiniteString}. 
As such, the Weyl--Titchmarsh function $m$ has an integral representation \cite{kakr74a}, \cite[Section~5.3]{roro94} of the form 
\begin{align}\label{eqnWTmIntRep}
 m(z) = c_1 z + c_2 +  \int_\R \frac{1}{\lambda-z} - \frac{\lambda}{1+\lambda^2}\, d\rho(\lambda), \quad z\in\C\backslash\R, 
\end{align}
for some constants $c_1$, $c_2\in\R$ with $c_1\geq0$ and a positive Borel measure $\rho$ on $\R$ with  
\begin{align}\label{eq:Poisson}
 \int_\R \frac{d\rho(\lambda)}{1+\lambda^2} < \infty.
\end{align}
In fact, the measure $\rho$ can be seen to be a {\em spectral measure} for a self-adjoint linear relation associated with~\eqref{eqnDE}; we only refer to~\cite{CHPencil} for more details.\footnote{For Sturm--Liouville spectral problems on the real line (or, more generally, on an interval with two singular endpoints), it is more common to work with a $2\times2$-matrix valued Weyl--Titchmarsh function. However, as it was emphasized by K.~Kodaira~\cite{kod} (see also~\cite{gezi07,kst} for further details), in certain cases it is enough to use a scalar Weyl--Titchmarsh function instead. The existence of a fundamental system of solutions like the one in Theorem~\ref{thmThetaPhi} is sufficient for example. Another instance of this is the case of Krein strings with a singular left endpoint as considered in~\cite{kot07}, which indicates that  the strong decay condition at $-\infty$ in Definition~\ref{def:Dspace} may be relaxed slightly. Let us also mention that a spectral measure can be introduced in a more straightforward way by employing the method of directing functionals (see, for instance,~\cite{fll08, kac56, kac65}).}

\begin{definition}\label{def:Spectrum}
The spectrum $\sigma$ of the differential equation~\eqref{eqnDE} is defined as the topological support of the measure $\rho$ or equivalently as the smallest closed set $B\subseteq\R$ such that $m$ has an analytic continuation to $\C\backslash B$. 
\end{definition}
 
Condition~\eqref{eqnMdef+} on the asymptotic behavior of the coefficients at $\infty$ entails that the spectrum $\sigma$ has a gap around zero. 
Moreover, we will see in the next result that the function $m$ is indeed uniquely determined by the spectral measure $\rho$. 

\begin{proposition}\label{prop:WTrepr}
  Zero does not belong to the spectrum $\sigma$ and 
\begin{align}\label{eqnWTmIntRepZero}
  m(z) = \int_\R \frac{z}{\lambda(\lambda-z)} d\rho(\lambda), \quad z\in\C\backslash\R. 
\end{align}
\end{proposition}

 \begin{proof}
    With the notation from the proof of Theorem~\ref{thmThetaPhi} we have 
 \begin{align*}
  m(z) = -\lim_{x\rightarrow\infty} \frac{\theta(z,x)}{z\phi(z,x)} = - \lim_{x\rightarrow\infty} \frac{Y_{11}(z,x)}{Y_{12}(z,x)},
\end{align*}
 which shows that $m$ is the Weyl--Titchmarsh function of the generalized indefinite string $(L,\wt{\omega},\wt{\dip})$ with $L=\infty$.  
 Because the pair $(\wt\Wr,\wt\dip)$ satisfies~\eqref{eqnMdefTilde}, it follows from~\cite[Theorem~5.2~(i)]{DSpec} that $m$ admits an analytic continuation to zero and that zero does not belong to the spectrum $\sigma$.
 Moreover, we also get that   
\begin{align}\label{eqnWratinf}
    \lim_{x\rightarrow\infty} \frac{1}{x}\int_0^{x}  \wt{\Wr}(s)ds = 0. 
\end{align}
By repeating the arguments in~\cite[Section~6]{AsymCS}, one can show that this limit actually coincides with $m(0)$.
For the sake of completeness however, we shall present the details here. 
To this end, let $r>0$ and consider the generalized indefinite string $(L,\wt\omega_r,\wt\dip_r)$ defined via the normalized anti-derivative $\wt\Wr_r$ of $\wt\omega_r$ by  
\begin{align*}
\wt{\Wr}_r(x) & = \wt{\Wr}(rx), & \wt{\dip}_r([0,x)) & = r^{-1}\wt{\dip}([0,rx)).
\end{align*} 
 Replacing $\wt{\Wr}$ and $\wt{\dip}$ in~\eqref{eqnCanSysString} with $\wt{\Wr}_r$ and $\wt{\dip}_r$, one sees that the corresponding matrix solution $Y_r$ is then simply given by
\begin{align*}
Y_r(z,x) = Y(r^{-1}z,rx),
\end{align*}
which implies that the corresponding Weyl--Titchmarsh function $m_r$ is given by
\begin{align*}
m_r(z) = m(r^{-1}z).
\end{align*}
%
On the other hand, we clearly have  
\begin{align*}
\lim_{r\to \infty}\int_0^x \wt{\Wr}_r(s)ds  = \lim_{r\to \infty}\frac{1}{r}\int_0^{rx} \wt{\Wr}(s)ds = 0
\end{align*}
for all $x>0$ in view of~\eqref{eqnWratinf}. 
Furthermore, denoting 
\begin{align*}
\varsigma_r(x) = \int_0^x \wt{\Wr}_r(s)^2 ds + \int_{0}^{x} d\wt{\dip}_r,\quad x\in [0,\infty),
\end{align*}
 we get that $\varsigma_r(x)\rightarrow0$ as $r\rightarrow\infty$ uniformly for all $x>0$ because 
\begin{align*}
 \varsigma_r(x) & = \frac{1}{r}\int_0^{rx} \wt{\Wr}(s)^2 ds + \frac{1}{r}\int_{0}^{rx} d\wt{\dip} \leq \frac{1}{r}\int_0^{\infty} \wt{\Wr}(s)^2 ds + \frac{1}{r}\int_{0}^{\infty} d\wt{\dip}.
\end{align*}
We conclude that the generalized indefinite strings $(L,\wt{\omega}_r, \wt{\dip}_r)$ converge to $(L,0,0)$ as $r\rightarrow\infty$ in the sense of \cite[Proposition~6.2]{IndefiniteString}.
 Since the Weyl--Titchmarsh function corresponding to $(L,0,0)$ is identically zero, it follows from~\cite[Proposition~6.2]{IndefiniteString} that $m_r$ converges locally uniformly to zero as $r\rightarrow\infty$, which proves that  
\begin{align*}
  m(0) = \lim_{r\rightarrow\infty} m(r^{-1}\I)=  \lim_{r\rightarrow\infty} m_r(\I) = 0. 
\end{align*}
It remains to observe that, by~\cite[Lemma~7.1]{IndefiniteString}, the constant $c_1$ is zero as $\wt{\dip}(\{0\}) = 0$. 
Because of this, we are able to compute $c_2$ in terms of $\rho$ by evaluating~\eqref{eqnWTmIntRep} at zero and taking into account that $m(0)=0$, which gives 
\[
c_2 
= -\int_{\R} \frac{d\rho(\lambda)}{\lambda(1+\lambda^2)}.
\] 
This readily yields the desired representation~\eqref{eqnWTmIntRepZero}.
 \end{proof}

 \begin{remark}\label{remMatZ}   
 It is also possible to express the derivative of $m$ at zero in terms of the coefficients of the spectral problem~\eqref{eqnDE}, which leads to the relation  
\begin{align}\label{eqnLCPars}
   \int_{\R} \E^{-x} (u(x) + u'(x))^2dx + \int_{\R} \E^{-x}d\dip(x) = \int_\R \frac{d\rho(\lambda)}{\lambda^2}.
  \end{align}
 Formulas of this type usually play an important role in understanding the corresponding completely integrable nonlinear equations because they allow to control solutions through the corresponding spectral data. 
We will give a proof of relation~\eqref{eqnLCPars} in Section~\ref{secIST} (see Corollary~\ref{cor:TraceFla} there). 
 \end{remark}

Our next aim is to show that the quantity in~\eqref{eqnMdef+} allows to control the size of the spectral gap around zero. 
To this end, we define  $E(u,\mu)$ by 
\begin{align}\label{eq:NormDef}
E(u,\mu) = \sup_{x\in \R}\, \E^{\frac{x}{2}} \biggl(\int_{x}^{\infty}\E^{-s}(u(s) + u'(s))^2ds + \int_{x}^{\infty}\E^{-s}d\dip(s)\biggr)^{\nicefrac{1}{2}}, 
\end{align}
which is finite due to the growth restrictions in~\eqref{eqnMdef-} and~\eqref{eqnMdef+}.
Furthermore, we define $\lambda_0(\sigma)$ as a measure for the size of the spectral gap by 
\begin{align}\label{def:lam0}
\lambda_0(\sigma) =  \inf_{\lambda\in \sigma} |\lambda|.
\end{align}

\begin{proposition}\label{prop:SpecGap}
   One has the estimates  
\begin{align}\label{eq:lam0est}
    \frac{1}{6\lambda_0(\sigma)} \leq E(u,\mu) \le \frac{\sqrt{2}}{\lambda_0(\sigma)}.
\end{align}
\end{proposition} 

\begin{proof}
  The two-sided estimate in the claim follows by employing the connection between the spectral problems~\eqref{eqnISP} and~\eqref{eqntildeString} as well as some considerations in~\cite{DSpec}. 
  Namely, since $\sigma$ coincides with the spectrum of the generalized indefinite string $(L,\wt\omega,\wt\dip)$, one needs to use~\cite[Theorem~4.6~(iv)]{DSpec} together with two-sided estimates on the norms of the integral operators there, which can be found in~\cite{chev70, muc, kakr58, AJPR} (the relation to the operators in~\cite[Theorem~4.6~(iv)]{DSpec} is explained in the proof of~\cite[Theorem~3.5]{DSpec}).
 These estimates give 
\begin{align*}
\frac{1}{6\wt{E}(\wt{\Wr},\wt{\dip})}\le \lambda_0(\sigma) \le \frac{\sqrt{2}}{\wt{E}(\wt{\Wr},\wt{\dip})}, 
\end{align*}
where $\wt{E}(\wt{\Wr},\wt{\dip})$ is defined by 
\begin{align*}
\wt{E}(\wt{\Wr},\wt{\dip}) = \sup_{x> 0}\, 
 \biggl(x\int_{x}^{\infty}\wt{\Wr}(s)^2ds + x\int_{x}^{\infty}d\wt{\dip}\biggr)^{\nicefrac{1}{2}}.
\end{align*}
 It remains to perform a simple change of variables to see that $\wt{E}(\wt{\Wr},\wt{\dip}) = E(u,\mu)$, which concludes the proof.
\end{proof}

\begin{remark}\label{remEssSpecGap}
  One can find a similar two-sided estimate for the gap of the essential spectrum around zero, by using~\cite[Theorem~4.6]{DSpec} again, together with the estimates in~\cite[Theorem~2]{st73}.
  The main difference hereby is that the supremum in~\eqref{eq:NormDef} has to be replaced by a limes superior at $\infty$. 
\end{remark}

 In this context, let us also note the following result that characterizes the subclass of $\CHdom$ that gives rise to purely positive spectrum.  

\begin{proposition}\label{propDefinite}
 The spectrum $\sigma$ is positive if and only if the measure $\dip$ vanishes identically and the distribution $\omega$ is a positive Borel measure on $\R$. 
 In this case, the measure $\omega$ is finite near $-\infty$ and $u'$ has a representative that is locally of bounded variation. 
 \end{proposition}
 
 \begin{proof}
   It again suffices to employ the connection with the spectral problem~\eqref{eqntildeString}. 
   Namely, by~\cite[Lemma~7.2]{IndefiniteString}, the spectral measure $\rho$ is supported on $[0,\infty)$ if and only if the measure $\wt{\dip}$ vanishes identically as well as that the function $\wt{\Wr}$ has a non-decreasing  representative.
  Clearly, the measure $\wt{\dip}$ vanishes identically if and only if so does $\dip$. 
 Given some $h\in H^1_{\cc}(\R)$, we define $h_{\Sr}\in H^1_\cc(0,\infty)$ by $h_{\Sr}(x) = x h(\log(x))$ so that one has, similar to~\eqref{eq:WrString}, that
  \begin{align*}
  \omega(h) =  -\int_0^\infty \wt{\Wr}(x) h_\Sr'(x)dx = \wt\omega(h_\Sr),
 \end{align*}
  which shows that $\omega$ is a positive distribution on $\R$ if and only if $\wt{\Wr}$ has a non-decreasing representative.
  Now fix an $x_0\in\R$ such that the derivative $u'(x_0)$ exists and for $k\in\N$ let $h_k\in H_{\cc}^1(\R)$ be the piecewise linear function such that $h_k$ is equal to one on $[-k,x_0-1/k]$, equal to zero outside of $[-k-1,x_0]$ and linear in between. 
 The remaining claims then follow because  
 \begin{align}\label{eqnOmegaInt}
 \begin{split} 
  \int_{-\infty}^{x_0} d\omega & = \lim_{k\rightarrow\infty} \int_\R h_k\, d\omega = \lim_{k\rightarrow\infty} \int_\R u(x) h_k(x)dx + \int_\R u'(x) h_k'(x)dx \\
      & = \int_{-\infty}^{x_0} u(x)dx - u'(x_0).
 \end{split}
 \end{align}
\end{proof} 
 
 \begin{remark}\label{remEumon}
    We have seen that the function $\wt{\Wr}$ has a non-decreasing representative when $\omega$ is a positive Borel measure on $\R$. 
    Because $\wt\Wr$ is square integrable, it follows that this representative is moreover non-positive and tends to zero at $\infty$. 
    This entails that in this case the function 
    \begin{align}\label{eqnEumon}
       \E^{-x}(u(x)+u'(x)) = -\wt\Wr(\E^x)
    \end{align}
    is non-increasing and non-negative on $\R$ with limit zero as $x\rightarrow\infty$. 
\end{remark}
 
  \begin{corollary}\label{corsigposS1}
  If the spectrum $\sigma$ is positive, then the following are equivalent:
  \begin{enumerate}[label=(\roman*), ref=(\roman*), leftmargin=*, widest=iii]
    \item\label{itmsigposi} The spectral measure $\rho$ satisfies 
 \begin{align}\label{eq:growth1rho}
 \int_0^\infty \frac{d\rho(\lambda)}{\lambda} < \infty.
 \end{align}
  \item\label{itmsigposii} The Weyl--Titchmarsh function $m(\I y)$ has a finite limit as $y\rightarrow\infty$. 
  \item\label{itmsigposiii} When $u'$ denotes the unique left-continuous representative, one has 
 \begin{align}\label{eq:CviaU}
   \limsup_{x\to -\infty}\E^{-x}(u(x) + u'(x)) < \infty.
 \end{align}
 \item\label{itmsigposiv} The measure $\omega$ satisfies
  \begin{align}
    \int_{-\infty}^0 \E^{-x} d\omega(x) < \infty. 
  \end{align}
 \end{enumerate}
  In this case, one has the relations
  \begin{align}\label{eqnsigposid}
     \int_0^\infty \frac{d\rho(\lambda)}{\lambda} = - \lim_{y\rightarrow\infty} m(\I y) = \lim_{x\rightarrow-\infty} \E^{-x}(u(x)+u'(x)) = \int_\R \E^{-x}d\omega(x).
  \end{align}
  \end{corollary}
 
 \begin{proof}
   If condition~\ref{itmsigposi} holds, then $m$ admits the integral representation
 \begin{align*}
  m(z) = -\int_0^\infty \frac{d\rho(\lambda)}{\lambda} + \int_0^\infty \frac{d\rho(\lambda)}{\lambda-z},  
\end{align*}
 from which one sees that condition~\ref{itmsigposii} follows with 
\begin{align*}
  \lim_{y\rightarrow\infty} m(\I y) =  - \int_0^\infty \frac{d\rho(\lambda)}{\lambda}. 
\end{align*}
 In view of \cite[Theorem~6.1]{AsymCS}, condition~\ref{itmsigposii} implies existence of the finite limit 
 \begin{align*}
  \lim_{x\to 0} \frac{1}{x}\int_0^x \wt{\Wr}(s)ds = \lim_{y\rightarrow\infty} \re\,m(\I y). 
\end{align*}
 Using~\eqref{eqnEumon} together with~\eqref{eqnOmegaInt},
 \begin{align*}
  \lim_{x\to 0} \frac{1}{x}\int_0^x \wt{\Wr}(s)ds =  -\lim_{x\rightarrow-\infty}\left(\E^{-x}\int_{-\infty}^x d\omega + \E^{-x}(u(x)+u'(x))\right),
\end{align*} 
 which entails~\eqref{eq:CviaU} and thus condition~\ref{itmsigposiii} because
  \begin{align*}
   \limsup_{x\to -\infty}\E^{-x}(u(x) + u'(x)) \leq \lim_{x\rightarrow-\infty} \left(\E^{-x}\int_{-\infty}^x d\omega + \E^{-x}(u(x)+u'(x))\right) < \infty.
 \end{align*}
  For $x$, $y\in\R$ we get using integration by parts and~\eqref{eqnOmegaInt} that 
  \begin{align}\label{eqnsigposeom}
     \int_x^y \E^{-s}d\omega(s) = \bigl.-\E^{-s}(u(s)+u'(s))\bigr|_{s=x}^y,
  \end{align}
  which proves that condition~\ref{itmsigposiii} is equivalent to condition~\ref{itmsigposiv}. 
  
  Finally, condition~\ref{itmsigposiii} implies that $\wt\Wr$ has a representative that is non-decreasing and bounded from below by a negative constant $-\wt\omega_0$. 
  This means that the corresponding distribution $\wt\omega$ is a real-valued Borel measure on $[0,\infty)$ that is positive on $(0,\infty)$ but with a negative point mass at zero. 
  As this negative point mass is bounded by $\wt\omega_0$, we infer that $\wt\omega+\wt\omega_0\delta_0$, where $\delta_0$ is the unit Dirac measure at zero, is a positive Borel measure on $[0,\infty)$. 
  The generalized indefinite string $(L,\wt\omega+\wt\omega_0\delta_0,0)$ then has the same spectral measure as $(L,\wt\omega,0)$ because adding a point mass at zero to $\wt\omega$ amounts to adding a real constant to the Weyl--Titchmarsh function which does not affect the measure in the integral representation in~\eqref{eqnWTmIntRep}; see~\cite[Remark~2.5]{ISPforCH}.
  Since $(L,\wt\omega+\wt\omega_0\delta_0,0)$ is actually a Krein string, the associated spectral measure $\rho$ satisfies~\eqref{eq:growth1rho} in view of \cite[Proposition~7.3]{IndefiniteString}. 
  For the second identity in~\eqref{eqnsigposid}, it remains to notice that 
  \begin{align*}
    \lim_{x\rightarrow 0} \frac{1}{x}\int_0^x \wt\Wr(s)ds = \lim_{x\rightarrow0} \wt\Wr(x) = - \lim_{x\rightarrow-\infty} \E^{-x}(u(x)+u'(x))
  \end{align*}
  because $\wt\Wr$ is non-decreasing and bounded from below.
  The last identity in~\eqref{eqnsigposid} follows by letting $x\rightarrow-\infty$ and $y\rightarrow\infty$ in~\eqref{eqnsigposeom} and using that 
  \begin{align*}
    \lim_{y\rightarrow\infty} \E^{-y}(u(y)+u'(y)) = -\lim_{y\rightarrow\infty} \wt\Wr(y) = 0,    
  \end{align*}
  which holds because $\wt\Wr$ is square integrable. 
\end{proof} 

  We are next going to characterize various subclasses of $\CHdom$ that give rise to purely discrete spectrum. 
  Because the spectrum is invariant under the conservative Camassa--Holm flow, these conditions on the coefficients are preserved too. 
  
\begin{proposition}\label{prop:Sp-classes}
 The following assertions hold:
\begin{enumerate}[label=(\roman*), ref=(\roman*), leftmargin=*, widest=iii]
  \item\label{prop:Sp-classes-i}  The spectrum $\sigma$ is purely discrete if and only if 
  \begin{align}
    \lim_{x\rightarrow\infty} \E^{x}\biggl(\int_{x}^{\infty}\E^{-s}(u(s) + u'(s))^2ds + \int_{x}^{\infty}\E^{-s}d\dip(s)\biggr) = 0.
  \end{align}
  \item\label{prop:Sp-classes-ii} For each $p>1$, the spectrum $\sigma$ satisfies 
  \begin{align}\label{eqnSinSpCH}
    \sum_{\lambda\in\sigma} \frac{1}{|\lambda|^p} < \infty
  \end{align}
  if and only if  
  \begin{align}\label{eq:umuSp}
    \int_{\R} \E^{\frac{px}{2}}\biggl(\int_{x}^{\infty}\E^{-s}(u(s) + u'(s))^2ds + \int_{x}^{\infty}\E^{-s}d\dip(s)\biggr)^{\nicefrac{p}{2}} dx < \infty.
  \end{align}
  \item\label{itmS2} The spectrum $\sigma$ satisfies~\eqref{eqnSinSpCH} with $p=2$  if and only if the measure $\mu$ is finite. 
  In this case, one has 
  \begin{align}\label{eq:traceS2}
    \frac{1}{2} \sum_{\lambda\in\sigma} \frac{1}{\lambda^2} = \int_\R d\mu = \|u\|^2_{H^1(\R)}  + \int_\R d\dip.
  \end{align}
  \item\label{itmS1} If the spectrum $\sigma$ satisfies~\eqref{eqnSinSpCH} with $p=1$, then 
  \begin{align}
 \int_{\R} \E^{\frac{x}{2}}\biggl(\int_{x}^{\infty}\E^{-s}(u(s) + u'(s))^2ds\biggr)^{\nicefrac{1}{2}} dx < \infty, 
\end{align}
 the function $u + u'$ is integrable with 
\begin{align}\label{eq:traceS1}
 \sum_{\lambda\in\sigma} \frac{1}{\lambda}  = \int_{\R} u(s) + u'(s) ds = \lim_{x\to \infty} \int_{-\infty}^x u(s)ds,
\end{align}
and the measure $\dip$ is singular with respect to the Lebesgue measure.
 \end{enumerate}
\end{proposition}

\begin{proof}
The claims again follow by using the connection with the spectral problem~\eqref{eqntildeString} and by applying the corresponding characterization from \cite[Section~5]{DSpec}. 
Let us only elaborate on~\ref{itmS2} and~\ref{itmS1}.
First of all, by \cite[Theorem~5.2~(iv)]{DSpec},
  \begin{align*}
    \frac{1}{2} \sum_{\lambda\in\sigma} \frac{1}{\lambda^2} = \int_0^\infty x\wt{\Wr}(x)^2 dx + \int_{0}^{\infty} x d\wt{\dip}(x) = \int_\R (u(x) + u'(x))^2dx + \int_\R d\dip,
  \end{align*}
 which in particular implies that $\dip$ is finite and that $u+u' \in L^2(\R)$. 
  Moreover, for all $x$, $y\in\R$ with $x<y$ one has 
  \begin{align*}
  \int_x^y u(s)^2 + u'(s)^2ds & =  \int_x^y (u(s) + u'(s))^2ds - u(y)^2 + u(x)^2 \\
  & \leq \|u+u'\|^2_{L^2(\R)} + u(x)^2,
  \end{align*} 
which shows that $u\in H^1(\R)$ as well as $\|u+u'\|_{L^2(\R)} = \|u\|_{H^1(\R)}$. 
Regarding~\ref{itmS1}, we only add that by~\ref{itmS2}, $u\in H^1(\R)$ and hence $u(x)\to 0$ as $|x|\to \infty$. Moreover, according to \eqref{eqnMdef-}, both $u$ and $u'$ are integrable near $-\infty$.
\end{proof}

\begin{remark}\label{rem:Besov}
Let us mention that the conditions on the function $u$ in Proposition~\ref{prop:Sp-classes} can be viewed from the perspective of Besov spaces. 
More specifically, the condition on $u$ in~\eqref{eq:umuSp} means that the Fourier transform of the corresponding function $\wt{\Wr}$, upon extending $\wt{\Wr}$ to the negative semi-axis by zero, belongs to the Besov space $B_{2\,p}^{\nicefrac{1}{2}}$ (see \cite[Remark on page~64]{AJPR} for example). 
However, it is possible to recover $u$ from $\wt\Wr$ with the help of~\eqref{eqnuitofa}, which is a composition of the Hardy operator and a change of variables.
 It remains unclear to us whether~\eqref{eq:umuSp} can be interpreted as some more familiar norm. 
For local well-posedness of the Camassa--Holm equation in Besov spaces we refer to~\cite{dat}. 
\end{remark}

\begin{remark}
In the recent paper~\cite{rowo20}, Romanov and Woracek obtained a characterization of $2\times2$ canonical systems whose resolvents belong to trace ideals enjoying Matsaev's property (in particular, so are Schatten--von Neumann ideals $\mathfrak{S}_p$ with $p>1$). Using the one-to-one correspondence between canonical systems and  generalized indefinite strings established in~\cite{IndefiniteString}, one may obtain further subclasses of $\CHdom$ with purely discrete spectrum and various summability properties, however, we are not going to pursue this goal here.  
\end{remark}

We finish this section by combining Proposition~\ref{prop:Sp-classes} with Corollary~\ref{corsigposS1}. 
It turns out that, under the additional positivity assumption, the corresponding characterizations in Proposition~\ref{prop:Sp-classes} become much more transparent.

\begin{corollary}\label{cor:Sp-classes}
If the spectrum $\sigma$ is positive and the equivalent conditions in Corollary~\ref{corsigposS1} are satisfied, then the following assertions hold:
\begin{enumerate}[label=(\roman*), ref=(\roman*), leftmargin=*, widest=iii]
  \item\label{itmSpplusi}  The spectrum $\sigma$ is purely discrete if and only if 
  \begin{align}\label{eq:u+u'=0}
    \lim_{x\rightarrow\infty} u(x)+u'(x) = 0.
  \end{align}
  \item\label{itmSpplusii} For each $p>\nicefrac{1}{2}$, the spectrum $\sigma$ satisfies 
  \begin{align}\label{eqnSinSpCHpos}
    \sum_{\lambda\in\sigma} \frac{1}{|\lambda|^p} < \infty
  \end{align}
  if and only if $u+u'\in L^p(\R)$. 
  \item\label{itmSpplusiii} The spectrum $\sigma$ satisfies~\eqref{eqnSinSpCHpos} with $p=1$ if and only if the measure $\omega$ is finite. 
  In this case, one has 
  \begin{align}\label{eqnTF1omega}
   \sum_{\lambda\in\sigma} \frac{1}{\lambda} = \int_\R d\omega.
  \end{align}
  \item\label{itmSpplusiv} If the spectrum $\sigma$ satisfies~\eqref{eqnSinSpCHpos} with $p=\nicefrac{1}{2}$, then the measure $\omega$ 
is singular with respect to the Lebesgue measure.
 \end{enumerate}
\end{corollary}

\begin{proof}
  Our assumptions ensure that $\wt\Wr$ has a representative that is non-decreasing and bounded from below by a constant $\wt\omega_0$. 
  The claims in~\ref{itmSpplusi},~\ref{itmSpplusii} and~\ref{itmSpplusiv} then follow by applying~\cite[Theorem~5.4]{DSpec} to the Krein string $(L,\wt\omega-\wt\omega_0\delta_0,0)$, which has the same spectrum as $(L,\wt\omega,0)$. 
  Here, we also use that $\wt{\Wr}$ is non-positive with limit zero at infinity, from which we infer that 
\begin{align*}
x\int_{x}^\infty d\wt{\omega} = - x\wt{\Wr}(x) = u(\log x) + u'(\log x).
\end{align*}
 Furthermore, we have seen in~\eqref{eqnOmegaInt} that 
 \begin{align*}
   \int_{-\infty}^x d\omega = \int_{-\infty}^x u(s)ds + u(x) - (u(x)+u'(x)).
 \end{align*}
 Using Proposition~\ref{prop:Sp-classes}~\ref{itmS1} and~\ref{itmS2} as well as~\eqref{eq:u+u'=0}, this shows that the measure $\omega$ is finite with~\eqref{eqnTF1omega} when the spectrum $\sigma$ satisfies~\eqref{eqnSinSpCHpos} with $p=1$. 
 The converse in~\ref{itmSpplusiii} follows from Mercer's theorem; see~\cite[Proposition~3.3]{IsospecCH}.
 \end{proof}

\begin{remark}
  If the spectrum $\sigma$ is positive, then we get from~\eqref{eqnsigposeom} that 
  \begin{align}
     u(x)+u'(x) = \E^x \int_x^\infty \E^{-s}d\omega(s).
  \end{align}
\end{remark}

\begin{remark}\label{remNotKreinReg}
 The assumption in Corollary~\ref{cor:Sp-classes} that the equivalent conditions in Corollary~\ref{corsigposS1} are satisfied is actually not necessary for most (but not all) of the claims. 
More precisely, this applies to those claims for which the conditions on the spectrum as well as the coefficients only depend on the asymptotic behavior near $\infty$, so that one can employ a standard interval splitting argument (Glazman's decomposition principle). 
\end{remark}

 \section{The spectral transform}\label{secIST}

We are now going to consider the {\em (direct) spectral transform} $(u,\mu)\mapsto\rho$ on $\CHdom$, where $\rho$ is the spectral measure as defined in Section~\ref{secSP}.  
 To this end, let us introduce the set $\SM_0$ of all positive Borel measures $\rho_0$ on $\R$ satisfying
 \begin{align}\label{eq:SM2}
 \int_\R \frac{d\rho_0(\lambda)}{1+\lambda^2} < \infty
\end{align}
 and whose topological support does not contain zero. 
   
 \begin{theorem}\label{thmIP}
   The mapping $(u,\mu)\mapsto\rho$ is a bijection from $\CHdom$ to $\SM_0$. 
 \end{theorem}

 \begin{proof}
  Since the spectral measure $\rho$ uniquely determines the Weyl--Titchmarsh function $m$ in view of Proposition~\ref{prop:WTrepr}, the uniqueness part in \cite[Theorem~6.1]{IndefiniteString} shows that the function $\wt{\Wr}$ and the measure $\wt{\dip}$ are uniquely determined as well. 
  It remains to apply Lemma~\ref{lemumuwdip} to see that the mapping in the claim is injective.
  
 In order to prove surjectivity, let $\rho$ be a measure in $\SM_0$ and introduce the Herglotz--Nevanlinna function (motivated by Proposition~\ref{prop:WTrepr})  
   \begin{align*}
    m(z) =  \int_{\R} \frac{z}{\lambda(\lambda-z)} d \rho(\lambda).
   \end{align*}
  From the existence part of \cite[Theorem~6.1]{IndefiniteString}, we obtain a real-valued and locally square integrable function $\wt{\Wr}$ on $[0,\infty)$ and a positive Borel measure $\wt{\dip}$ on $[0,\infty)$ with $\wt{\dip}(\lbrace0\rbrace)=0$ such that the function $m$ is the Weyl--Titchmarsh function for the system~\eqref{eqnCanSysString}, that is, if $Y(z,\redot)$ denotes the unique solution of the integral equation~\eqref{eqnCanSysString} for each $z\in\C$, then the function $m$ is given by
 \begin{align*}
  m(z) = -\lim_{x \rightarrow\infty} \frac{Y_{11}(z, x)}{Y_{12}(z, x)}, \quad z\in\C\backslash\R. 
 \end{align*}
 However, since the support of $\rho$ does not contain zero, the function $m$ is analytic at zero. 
 By~\cite[Theorem~5.2~(i)]{DSpec}, condition~\eqref{eqnMdefTilde} is satisfied and hence the pair $(\wt\Wr,\wt\dip)$ belongs to $\String_0$. 
 It remains to apply Lemma~\ref{lemumuwdip} once again to conclude that the corresponding pair $(u,\mu)$, constructed by means of~\eqref{eqnuitofa} and~\eqref{eqnDefbeta}, belongs to $\CHdom$ and has $\rho$ as its corresponding spectral measure.  
 \end{proof}

It is not difficult to see that bijectivity of the spectral transform is preserved upon restricting to suitable subclasses of $\CHdom$ and $\SM_0$, in particular to ones that arise from the characterizations of properties in the previous section. 
Let us provide here only one such example. 
Namely, motivated by Proposition~\ref{propDefinite}, we introduce the following two sets:
\begin{enumerate}[label=\textbullet, leftmargin=*]
\item $\CHdom^+$ consisting of all pairs $(u,\mu)$ from $\CHdom$ such that $\omega$ is a positive Borel measure on $\R$ and $\dip$ vanishes identically. 
\item $\SM_0^+$ consisting of all measures from $\SM_0$  whose support is contained in $(0,\infty)$. 
\end{enumerate}
Combining Theorem~\ref{thmIP} with  Proposition~\ref{propDefinite}, we immediately end up with the following result.
   
 \begin{corollary}\label{corIP+}
 The mapping $(u,\mu)\mapsto \rho$ is a bijection from $\CHdom^+$ to $\SM_0^+$. 
 \end{corollary} 
 
 The set $\CHdom^+$ indeed admits a very transparent description in terms of $u$.
  
\begin{lemma}\label{lem:D+}
The mapping $(u,\mu)\mapsto u$ is a bijection between $\CHdom^+$ and the set of all real-valued functions $u$ in $H^1_{\loc}(\R)$ such that the distribution $u-u''$ is a positive Borel measure on $\R$, the function $u+u'$ belongs to $L^\infty(\R)$ and 
   \begin{align}
     \int_{-\infty}^0 \E^{-x}\bigl(u(x)^2 + u'(x)^2\bigr) dx <\infty.
   \end{align} 
 Moreover, for pairs $(u,\mu)$ in $\CHdom^+$ one has the estimates 
\begin{align}\label{eq:SpecGapviau+u'}
\frac{1}{6\lambda_0(\sigma)} \le \|u+u'\|_{L^\infty(\R)} \le \frac{2\sqrt{2}}{\lambda_0(\sigma)}.
\end{align}
\end{lemma}

\begin{proof}
Clearly, each pair $(u,\mu)$ in $\CHdom^+$ is uniquely determined by the function $u$ alone since the measure $\dip$ vanishes identically. 
Then bijectivity follows readily from Lemma~\ref{lem:umuEST} and Proposition~\ref{propDefinite}, except for the fact that the function $u+u'$ belongs to $L^\infty(\R)$ for every pair $(u,\mu)$ in $\CHdom^+$. 
However, the latter is immediate from the two-sided estimate 
  \begin{align}\label{eq:estEviau+u'}
   \frac{1}{2} \|u+u'\|_{L^\infty(\R)} \leq E(u,\mu)  \leq \|u+u'\|_{L^\infty(\R)}.
  \end{align}
The second inequality is trivial and the first one follows from condition~\eqref{eqnMdef+} by taking into account that the function~\eqref{eqnEumon} in Remark~\ref{remEumon} is non-increasing and non-negative in this case. 
 Indeed, by monotonicity we get
   \begin{align*}
   E(u,\mu)^2 & = \sup_{x\in\R} \E^{x}\int_x^\infty \E^{-s}(u(s)+u'(s))^2ds \\
   & \ge \sup_{x \leq x_0} \E^{x}\int_x^{x_0} \E^{s} \frac{(u(x_0)+u'(x_0))^2}{\E^{2x_0}}ds  = \frac{1}{4} (u(x_0)+u'(x_0))^2
   \end{align*}
   for all  $x_0\in\R$, which implies the desired estimate. 
   Finally, the estimates in~\eqref{eq:SpecGapviau+u'} are a straightforward combination of~\eqref{eq:estEviau+u'} with~\eqref{eq:lam0est}.  
\end{proof}
 
 \begin{remark}
 The constants in~\eqref{eq:SpecGapviau+u'} are not optimal and can be improved by using the connections with Krein strings and the spectral gap estimates from~\cite{kakr58}, \cite{muc}.
 \end{remark}

 We continue this section by establishing a continuity property for the spectral transform.   
 In order to state it, let $(u,\mu)$ be a pair in $\CHdom$ and $(u_k,\mu_k)$ be a sequence of pairs in $\CHdom$.
 All quantities corresponding to $(u,\mu)$ will be denoted as in Section~\ref{secSP} and the ones corresponding to $(u_k,\mu_k)$ with an additional subscript. 
 
  \begin{proposition}\label{propCont}
  The following are equivalent:
    \begin{enumerate}[label=(\roman*), ref=(\roman*), leftmargin=*, widest=iii]
    \item The pairs $(u_k,\mu_k)$ converge to $(u,\mu)$ in $\CHdom$. 
    \item\label{itmContii} The Weyl--Titchmarsh functions $m_k$ converge locally uniformly to $m$. 
    \item\label{itmContiii} For all continuous functions $\chi\in C(\R)$ such that the limit of $\chi(\lambda)$ as $|\lambda|\rightarrow\infty$ exists and is finite one has
  \begin{align}\label{eqnContSD2}
  \int_{\R} \frac{\chi(\lambda)}{1+\lambda^2} d\rho_k(\lambda) & \rightarrow \int_{\R} \frac{\chi(\lambda)}{1+\lambda^2} d\rho(\lambda), & \int_{\R} \frac{d\rho_k(\lambda)}{\lambda(1+\lambda^2)}  & \rightarrow \int_{\R} \frac{d\rho(\lambda)}{\lambda(1+\lambda^2)}.
 \end{align}
   \end{enumerate}
   \end{proposition}

 \begin{proof}
Observe first that~\ref{itmContii} and~\ref{itmContiii} are equivalent for Herglotz--Nevanlinna functions of the present type. 
Moreover, it follows from~\cite[Proposition~6.2]{IndefiniteString} that~\ref{itmContii} holds if and only if the corresponding pairs $(\wt\Wr_k,\wt\dip_k)$ converge to $(\wt\Wr,\wt\dip)$ in the sense of Lemma~\ref{lem:topStr}.
However, Lemma~\ref{lem:topStr} says that this is equivalent to convergence of the pairs $(u_k,\mu_k)$ to $(u,\mu)$ in $\CHdom$. 
 \end{proof}
 
 \begin{remark}
  Since locally uniform convergence on the set of Herglotz--Nevanlinna functions is induced by a metric, one can use the injection $(u,\mu)\mapsto m$ to transfer this metric to $\CHdom$. 
  In view of Proposition~\ref{propCont}, this metric then indeed induces our mode of convergence.
 \end{remark}

 For some results in the following, it will be necessary to restrict to subsets of $\CHdom$ on which the functional $E$ defined by~\eqref{eq:NormDef} is bounded. 
 To this end, we introduce the subsets $\CHdom(R)$ of $\CHdom$ by setting
 \begin{align}\label{eqnCHdomR}
 \CHdom(R) = \{(u,\mu)\in \CHdom\,|\, E(u,\mu) \leq R\}
 \end{align}
 for $R>0$.
 Due to the estimates in Proposition~\ref{prop:SpecGap}, pairs $(u,\mu)$ in $\CHdom(R)$ have a uniform spectral gap around zero. 
 In particular, this allows us to simplify condition~\ref{itmContiii} in Proposition~\ref{propCont}. 
 
 \begin{corollary}\label{corCont}
  If there is an $R>0$ such that $(u_k,\mu_k)$ belongs to $\CHdom(R)$ for all $k\in\N$, then the conditions in Proposition~\ref{propCont} are further equivalent to the following:
       \begin{enumerate}[label=(\roman*), ref=(\roman*), leftmargin=*, widest=iii]
        \setcounter{enumi}{3}\item\label{itmContiv} For all continuous functions $\chi\in C(\R)$ such that the limit of $\chi(\lambda)$ as $|\lambda|\rightarrow\infty$ exists and is finite one has
  \begin{align}\label{eqnContSD2R}
  \int_{\R} \frac{\chi(\lambda)}{1+\lambda^2} d\rho_k(\lambda) & \rightarrow \int_{\R} \frac{\chi(\lambda)}{1+\lambda^2} d\rho(\lambda).
 \end{align}
       \end{enumerate}
 \end{corollary}
 
  It also follows that the functional in~\eqref{eqnLCPars} is continuous on these subsets. 

\begin{lemma}\label{lem:convonR}
If the pairs $(u_k,\mu_k)$ converge to $(u,\mu)$ in $\CHdom$ and there is an $R>0$ such that $(u_k,\mu_k)$ belongs to $\CHdom(R)$ for all $k\in\N$, then
\begin{align}\label{eqnContMUonR}
       \int_{\R} \E^{- x}u_k(x)^2dx + \int_{\R} \E^{- x}d\mu_k(x)   \rightarrow \int_{\R} \E^{- x}u(x)^2dx+ \int_{\R}  \E^{- x}d\mu(x).
\end{align}
\end{lemma}

\begin{proof}
  First of all, we note that~\eqref{eqnContMUonR} becomes 
\begin{align*}
 \int_{\R} \E^{-s}(u_k(s) + u_k'(s))^2ds + \int_{\R}\E^{-s}d\dip_k(s) \to  \int_{\R} \E^{-s} (u(s) + u'(s))^2ds + \int_{\R} \E^{-s}d\dip(s)
\end{align*} 
 after an integration by parts. 
 Similarly, it follows from an integration by parts that
\begin{align*}
 &  \int_{-\infty}^x \E^{-s}(u_k(s) + u_k'(s))^2ds + \int_{-\infty}^x\E^{-s}d\dip_k(s) \\
 & \qquad\qquad = \int_{-\infty}^x \E^{- s}u_k(s)^2ds + \E^{-x}u_k(x)^2 + \int_{-\infty}^x \E^{- s}d\mu_k(s). 
\end{align*}
By the definition of convergence on $\CHdom$, as well as Lemma~\ref{lem:topD}, we then get  
\begin{align*}
 &  \int_{-\infty}^x \E^{-s}(u_k(s) + u_k'(s))^2ds + \int_{-\infty}^x\E^{-s}d\dip_k(s) \\
  & \qquad\qquad \rightarrow  \int_{-\infty}^x \E^{-s}(u(s) + u'(s))^2ds + \int_{-\infty}^x\E^{-s}d\dip(s)
\end{align*}
for all $x\in\R$.
 On the other side, since all pairs $(u_k,\mu_k)$ belong to $\CHdom(R)$ for some $R>0$, we also have the uniform bound 
\begin{align*}
 \int_{x}^{\infty}\E^{-s}(u_k(s) + u_k'(s))^2ds + \int_{x}^{\infty}\E^{-s}d\dip_k(s) \le R^2\E^{-x} 
\end{align*}
for all $x\in \R$ in view of~\eqref{eq:NormDef}. 
 Together, this ensures the claimed convergence.
\end{proof}

We will next employ this to prove locally uniform convergence of the corresponding functions $P_k$ as defined in~\eqref{eq:P0def} under the assumptions of Lemma~\ref{lem:convonR}. 

\begin{corollary}\label{cor:convPk}
If the pairs $(u_k,\mu_k)$ converge to $(u,\mu)$ in $\CHdom$ and there is an $R>0$ such that $(u_k,\mu_k)$ belongs to $\CHdom(R)$ for all $k\in\N$, then the functions $P_k$ converge locally uniformly to $P$.
\end{corollary}

\begin{proof}
By writing each function $P_k$ as in~\eqref{eq:P0defSum}, one sees that the first two summands converge almost everywhere (compare the proof of Lemma~\ref{lem:topD}~\ref{itmtopDi} and~\ref{itmtopDii}). 
We then use Lemma~\ref{lem:topD} together with Lemma~\ref{lem:convonR} to conclude that the remaining two summands also converge almost everywhere.
Since the functions $P_k$ are locally uniformly bounded (in view of Lemma~\ref{lem:convonR}) as well as locally uniformly equicontinuous, the claim follows by applying the Arzel\`a--Ascoli theorem. 
\end{proof}

An important role is played by particular kinds of pairs in $\CHdom$; so-called {\em multi-peakon profiles}. 
More specifically, a pair $(u,\mu)$ in $\CHdom$ is called a {\em multi-peakon profile} if $\omega$ and $\dip$ are Borel measures that are supported on a finite set. 
This entails that the function $u$ is of the form
\begin{align}\label{eq:MPprofile}
u(x) & = \frac{1}{2}\sum_{n=1}^N p_n \E^{-|x-x_n|}  
\end{align}
for some $N\in\N\cup\{0\}$,  $x_1,\ldots,x_N\in\R$ and $p_1,\ldots,p_N\in\R$. 
We shall denote the set of all multi-peakon profiles by $\Peakons$. 
A pair $(u,\mu)$ in $\CHdom$ belongs to $\Peakons$ if and only if the corresponding spectrum $\sigma$ is a finite set; see~\cite[Section~4]{ConservMP} and also the next remark. 
In this case, the pair $(u,\mu)$ can be recovered explicitly in terms of the moments of the corresponding spectral measure $\rho$; see~\cite[Corollary~4.6]{ConservMP}.

\begin{remark}
Using the bijection between the spaces $\CHdom$ and $\String_0$ in Lemma~\ref{lemumuwdip}, one sees that a pair $(u,\mu)$ in $\CHdom$ is a multi-peakon profile if and only if the function $\wt\Wr$ is piecewise constant (with finitely many steps) and the measure $\wt\dip$ has finite support.
In this case, the corresponding spectral problem~\eqref{eqntildeString} allows a complete direct and inverse spectral theory, which goes back to work of M.\ G.\ Krein and H.\ Langer on the indefinite moment problem~\cite{krla79} (see~\cite{IndMoment} and~\cite{StieltjesType} for further details).    
\end{remark}

 \begin{proposition}\label{prop:Peakons}
  For every pair $(u,\mu)$ in $\CHdom$ there is a sequence $(u_k,\mu_k)$ of multi-peakon profiles in $\Peakons$ with the following properties: 
      \begin{enumerate}[label=(\roman*), ref=(\roman*), leftmargin=*, widest=iii]
        \item The sequence $(u_k,\mu_k)$ converges to $(u,\mu)$ in $\CHdom$. 
        \item The corresponding spectral gaps satisfy $\lambda_0(\sigma_k)\geq\lambda_0(\sigma)$.
        \item The corresponding spectral measures satisfy 
  \begin{align}\label{eqnPdenserho}
      \int_\R \frac{d\rho_k(\lambda)}{\lambda^2}  \leq \int_\R \frac{d\rho(\lambda)}{\lambda^2}.
  \end{align} 
    \end{enumerate}
  If the pair $(u,\mu)$ belongs to $\CHdom^+$, then this sequence $(u_k,\mu_k)$ can be chosen in $\CHdom^+$. 
 \end{proposition}

 \begin{proof}
  Every measure $\rho$ in $\SM_0$ can be approximated in the sense of Corollary~\ref{corCont}~\ref{itmContiv} by a sequence of measures $\rho_k$ in $\SM_0$ supported on finite sets such that the massless gap $(-\eta,\eta)$ of $\rho$ around zero is preserved and~\eqref{eqnPdenserho} is satisfied. 
  For example, these measures can be chosen as
  \begin{align*}
    \rho_k = \sum_{j=1}^{k^2} \rho\biggl(\biggl(-\eta-\frac{j}{k},-\eta-\frac{j-1}{k}\biggr]\biggr) \delta_{-\eta-\frac{j}{k}} + \sum_{j=1}^{k^2} \rho\biggl(\biggl[\eta+\frac{j-1}{k},\eta+\frac{j}{k}\biggr)\biggr) \delta_{\eta+\frac{j}{k}}.
  \end{align*}
  If $\rho$ belongs to $\SM_0^+$, then this sequence also belongs to $\SM_0^+$. 
  The claim then follows by applying Proposition~\ref{propCont} and Corollary~\ref{corCont} because each measure $\rho_k$ corresponds to a multi-peakon profile $(u_k,\mu_k)$; see~\cite{ConservMP} or~\cite{besasz00} for the positive case. 
  Also note that it follows from Proposition~\ref{prop:SpecGap} that all the pairs $(u_k,\mu_k)$ belong to $\CHdom(R)$ for some $R>0$ due to the uniform spectral gaps. 
 \end{proof}

As an application of the above continuity and approximation properties, we shall now verify relation~\eqref{eqnLCPars}.

\begin{corollary}\label{cor:TraceFla}
  If $(u,\mu)$ is a pair in $\CHdom$ and $\rho$ is the corresponding spectral measure, then 
\begin{align}\label{eqnLCPars02}
\int_\R \E^{-x}u(x)^2dx + \int_\R \E^{-x}d\mu(x) = \int_\R \frac{d\rho(\lambda)}{\lambda^2}.
\end{align}
\end{corollary}

\begin{proof}
 We first note that an integration by parts shows that the left-hand sides in~\eqref{eqnLCPars02} and~\eqref{eqnLCPars} coincide.
 Let $(u_k,\mu_k)$ be an approximating sequence of multi-peakon profiles as in Proposition~\ref{prop:Peakons}. 
 For these pairs, the identity 
 \begin{align*}
\int_\R \E^{-x}u_k(x)^2dx + \int_\R \E^{-x}d\mu_k(x) = \int_\R \frac{d\rho_k(\lambda)}{\lambda^2}
\end{align*}
 holds according to~\cite[Proposition~2.6]{ConservCH}. 
 Due to the uniform spectral gap and the estimates in Proposition~\ref{prop:SpecGap}, we may apply Lemma~\ref{lem:convonR} to conclude that  
\begin{align*}
\int_\R \E^{-x}u_k(x)^2dx + \int_\R \E^{-x}d\mu_k(x) \to \int_\R \E^{-x}u(x)^2dx + \int_\R \E^{-x}d\mu(x).
\end{align*} 
 On the other hand, it follows from Proposition~\ref{propCont}~\ref{itmContiii} that   
\begin{align*}
\int_\R \frac{d\rho_k(\lambda)}{\lambda^2}  \to \int_\R \frac{d\rho(\lambda)}{\lambda^2},
\end{align*}
 where we again used that our sequence has a uniform spectral gap.
\end{proof}

\begin{remark}
As we will see in the next section, relation~\eqref{eqnLCPars02} is important in extending the conservative multi-peakon flow to the conservative Camassa--Holm flow on the whole space $\CHdom$. 
Let us also point out that the quantity appearing under the integral on the left-hand side is nothing but $u^2 + \mu$, which also defines the nonlocal term $P$ in~\eqref{eq:PequDef}.
\end{remark}

   For a fixed measure $\rho_0$ in $\SM_0$, we define the {\em isospectral set} $\Iso{\rho_0}$ as the set of all those pairs $(u,\mu)$ in $\CHdom$ whose associated spectral measure $\rho$ is mutually absolutely continuous with respect to $\rho_0$.
  It is an immediate consequence of Theorem~\ref{thmIP} that the isospectral set $\Iso{\rho_0}$ is in one-to-one correspondence with the set $\Lambda(\rho_0)$ consisting of all functions $\vartheta$ in $L^1_{\loc}(\R;\rho_0)$ that are positive almost everywhere with respect to $\rho_0$ and such that 
 \begin{align}\label{eqnLambdarho}
   \int_\R  \frac{\vartheta(\lambda)}{\lambda^2} d\rho_0(\lambda) <\infty,
 \end{align} 
 by means of the bijection given by the Radon--Nikodym derivative 
 \begin{align}\label{eqnIsoSpecTrans}
  (u,\mu) \mapsto  \frac{d\rho}{d\rho_0}.
 \end{align}
 In the special case when the topological support $\sigma_0$ of the measure $\rho_0$ is discrete, the set $\Iso{\rho_0}$ consists precisely of all pairs $(u,\mu)$ in $\CHdom$ whose spectrum coincides with $\sigma_0$. 
Under this assumption, the isospectral set $\Iso{\rho_0}$ can be  identified with 
  \begin{align}\label{eqnLambdasigma}
    \biggl\lbrace \gamma:\sigma_0\rightarrow(0,\infty) \,\biggl|\; \sum_{\lambda\in\sigma_0}  \frac{\gamma(\lambda)}{\lambda^2}  < \infty \biggr.\biggr\rbrace
 \end{align} 
 by means of the bijection $ (u,\mu) \mapsto\gamma$ given by $\gamma(\lambda) = \rho(\{\lambda\})$. 
The values of $\gamma$ here are precisely the reciprocals of the usual norming constants.

\section{The conservative Camassa--Holm flow}\label{secCCH}

 Let us define the {\em conservative Camassa--Holm flow} $\Phi$ on $\CHdom$  as a mapping
 \begin{align}
   \Phi\colon  \CHdom\times\R \rightarrow\CHdom
 \end{align}
 in the following way:
 Given a pair $(u,\mu)$ in $\CHdom$ with associated spectral measure $\rho$ (by means of the spectral transform considered in the previous section) as well as a $t\in\R$, the corresponding image $\Phi^t(u,\mu)$ under $\Phi$ is defined as the unique pair in $\CHdom$ for which the associated spectral measure is given by    
   \begin{align}\label{eqnSMEvo}
    B \mapsto \int_B \E^{-\frac{t}{2\lambda}} d\rho(\lambda)
  \end{align}
  on the Borel subsets of $\R$. 
 We note that $\Phi^t(u,\mu)$ is well-defined since the measure given by~\eqref{eqnSMEvo} belongs to $\SM_0$ whenever so does $\rho$ and hence the existence of a unique corresponding pair in $\CHdom$ is guaranteed by Theorem~\ref{thmIP}. 
  The definition of this flow is of course motivated by the well-known time evolution of spectral data for spatially decaying classical solutions of the Camassa--Holm equation as well as multi-peakons; see \cite[Section~6]{besasz98}.
   In fact, it reduces to the conservative multi-peakon flow in \cite{hora07a, ConservMP} when restricted to the set $\Peakons$.
 
   \begin{remark}
    For a fixed measure $\rho_0$ in $\SM_0$, the isospectral set $\Iso{\rho_0}$ is clearly invariant under the conservative Camassa--Holm flow.
    With respect to the bijection in~\eqref{eqnIsoSpecTrans}, the conservative Camassa--Holm flow on $\Iso{\rho_0}$ turns into the simple linear flow on $\Lambda(\rho_0)$ given by
  \begin{align}
    \vartheta(\lambda,t) = \E^{-\frac{t}{2\lambda}} \vartheta(\lambda,0).
  \end{align}
  Obviously, it becomes an equally simple linear flow on the set in~\eqref{eqnLambdasigma} when the topological support $\sigma_0$ of the measure $\rho_0$ is discrete.  
  More precisely, in this case one has 
   \begin{align}
    \gamma(\lambda,t) = \E^{-\frac{t}{2\lambda}} \gamma(\lambda,0)
  \end{align}
  for all eigenvalues $\lambda\in\sigma_0$. 
  This coincides with the known time evolution of norming constants for conservative multi-peakons \cite{besasz00, ConservMP} when $\sigma_0$ is a finite set. 
 \end{remark}  
 
  Before we prove the following continuity property of the conservative Camassa--Holm flow, let us point out that the sets $\CHdom(R)$ defined in~\eqref{eqnCHdomR} are not invariant. 
 
 \begin{proposition}\label{propFlowCont}
  The conservative Camassa--Holm flow $\Phi$ is continuous when restricted to $\CHdom(R)\times\R$ for every $R>0$.  
 \end{proposition}

 \begin{proof}
  Let $t_k\in\R$ be a sequence that converges to $t$ and suppose that the sequence $(u_k,\mu_k)$ of pairs in $\CHdom(R)$ converges to $(u,\mu)$ in $\CHdom$ so that Proposition~\ref{propCont}~\ref{itmContiii} holds for the corresponding spectral measures. 
  From the two-sided estimate in Proposition~\ref{prop:SpecGap}, we infer that the corresponding spectra have a uniform gap around zero with $\lambda_0(\sigma_k)\geq   1/(6R) =: \varepsilon$ and that the images $\Phi^{t_k}(u_k,\mu_k)$ belong to $\CHdom(9R)$ for all $k\in\N$.
 Given any continuous function $\chi\in C(\R)$ such that the limit of $\chi(\lambda)$ as $|\lambda|\rightarrow\infty$ exists and is finite, we choose a continuous function $\tau \in C(\R)$ such that 
  \begin{align*}
  \tau(\lambda) =  \chi(\lambda) \E^{-\frac{t}{2\lambda}}   
 \end{align*}
 when $|\lambda|\geq\varepsilon$, as well as a constant $K\in\R$ such that $|\tau(\lambda)|\leq K$ for all $\lambda\in\R$. 
 One then may estimate 
     \begin{align*} 
  & \biggl|  \int_\R \frac{\chi(\lambda)}{1+\lambda^2}\E^{-\frac{t_k}{2\lambda}} d\rho_k(\lambda)  -  \int_\R   \frac{\chi(\lambda)}{1+\lambda^2} \E^{-\frac{t}{2\lambda}} d\rho(\lambda) \biggr| \\
   & \qquad \leq   K  \frac{|t_k-t|}{2\varepsilon} \E^{\frac{|t_k-t|}{2\varepsilon}} \int_\R \frac{d\rho_k(\lambda)}{1+\lambda^2}   + 
   \biggl|\int_\R \frac{\tau(\lambda)}{1+\lambda^2} d\rho_k(\lambda)  - 
  \int_\R   \frac{\tau(\lambda)}{1+\lambda^2} d\rho(\lambda) \biggr|
 \end{align*}
 for every $k\in\N$.
 Since it follows readily from Proposition~\ref{propCont}~\ref{itmContiii} that the right-hand side converges to zero as $k\rightarrow\infty$, we infer that  
    \begin{align*}
  \int_\R   \frac{\chi(\lambda)}{1+\lambda^2} \E^{-\frac{t_k}{2\lambda}} d\rho_k(\lambda) & \rightarrow \int_\R   \frac{\chi(\lambda)}{1+\lambda^2} \E^{-\frac{t}{2\lambda}} d\rho(\lambda). 
 \end{align*}
   In view of Corollary~\ref{corCont}, this implies that $\Phi^{t_k}(u_k,\mu_k)$ converges to $\Phi^t(u,\mu)$ in $\CHdom$. 
 \end{proof}

In order to state the main results of this section, let us first specify a precise meaning for weak solutions to the two-component Camassa--Holm system~\eqref{eqnOurCH}.

\begin{definition}\label{defGCS}
A {\em global conservative solution} of the two-component Camassa--Holm system~\eqref{eqnOurCH} with initial data $(u_0,\mu_0)\in\CHdom$ is a continuous curve 
\begin{align}
  \gamma\colon t\mapsto(u(\ledot,t),\mu(\ledot,t))
\end{align}
from $\R$ to $\CHdom$ with $\gamma(0)=(u_0,\mu_0)$ that satisfies~\eqref{eqnOurCH} in the sense that for every test function $\varphi\in C_\cc^\infty(\R\times\R)$ one has  
 \begin{align}\label{eqnCHsysweak1}
 & \int_\R \int_\R u(x,t) \varphi_t(x,t) + \biggl(\frac{u(x,t)^2}{2} + P(x,t) \biggr) \varphi_x(x,t) \,dx \,dt = 0, \\
 \begin{split} 
 &  \int_\R \int_\R \varphi_t(x,t) + u(x,t) \varphi_x(x,t) \,d\mu(x,t) \,dt  \\ 
 &   \qquad\qquad\qquad\qquad = 2\int_\R \int_\R u(x,t)\biggl(\frac{u(x,t)^2}{2} - P(x,t) \biggr) \varphi_x(x,t) \,dx \,dt,
 \end{split}
 \end{align}
 where the function $P$ on $\R\times\R$ is given by 
 \begin{align}\label{eq:Pdef}
  P(x,t) =  \frac{1}{4} \int_\R \E^{-|x-s|} u(s,t)^2 ds +  \frac{1}{4} \int_\R \E^{-|x-s|} d\mu(s,t).
 \end{align}
 \end{definition}
 
 Let us point out that continuity of the curve $\gamma$ guarantees that the function $u$ is at least continuous on $\R\times\R$ in view of Lemma~\ref{lem:topD}. 
 
\begin{theorem}\label{thmWeakSol}
  For every pair $(u_0,\mu_0)\in\CHdom$, the integral curve $t\mapsto \Phi^t(u_0,\mu_0)$ is a global conservative solution of the two-component Camassa--Holm system~\eqref{eqnOurCH} with initial data $(u_0,\mu_0)$.
 \end{theorem} 

 \begin{proof}
  We first note that the integral curve $t\mapsto\Phi^t(u_0,\mu_0)$ is continuous by Proposition~\ref{propFlowCont} and that $\Phi^0(u_0,\mu_0)=(u_0,\mu_0)$ by definition. 
   The remaining claim is known if the pair $(u_0,\mu_0)$ is a multi-peakon profile; see~\cite[Section~4]{ConservCH} or~\cite[Section~5]{ConservMP}.
 Otherwise, let $(u_{k,0},\mu_{k,0})$ be a sequence of multi-peakon profiles that approximate the initial data $(u_0,\mu_0)$ in the sense of Proposition~\ref{prop:Peakons}. 
The corresponding pairs 
\begin{align*}
  (u_k(\ledot,t),\mu_k(\ledot,t)) = \Phi^t(u_{k,0},\mu_{k,0})
  \end{align*}
   obtained from the integral curves $t\mapsto \Phi^t(u_{k,0},\mu_{k,0})$ then satisfy   
 \begin{align}\label{eqnWSk1}
 & \int_\R \int_\R u_k(x,t) \varphi_t(x,t) + \biggl(\frac{u_k(x,t)^2}{2} + P_k(x,t) \biggr) \varphi_x(x,t) \,dx \,dt = 0, \\
 \begin{split} \label{eqnWSk2}
 &  \int_\R \int_\R \varphi_t(x,t) + u_k(x,t) \varphi_x(x,t) \,d\mu_k(x,t) \,dt  \\ 
 &   \qquad\qquad\qquad\qquad = 2\int_\R \int_\R u_k(x,t)\biggl(\frac{u_k(x,t)^2}{2} - P_k(x,t) \biggr) \varphi_x(x,t) \,dx \,dt,
 \end{split}
 \end{align}
 for every test function $\varphi\in C_\cc^\infty(\R\times\R)$, where the function $P_k$ is given by 
  \begin{align*}
  P_k(x,t) =  \frac{1}{4} \int_\R \E^{-|x-s|} u_k(s,t)^2 ds +  \frac{1}{4} \int_\R \E^{-|x-s|} d\mu_k(s,t).
 \end{align*}
 Lemma~\ref{lem:umuEST}, the relation~\eqref{eqnLCPars02} in Corollary~\ref{cor:TraceFla}, our definition of the flow and Proposition~\ref{prop:Peakons} guarantee the uniform bounds
 \begin{align*}
  \E^{-x}u_k(x,t)^2  \leq \int_\R \E^{-s} d\mu_k(s,t) \leq \int_\R \E^{-\frac{t}{2\lambda}} \frac{d\rho_{k,0}(\lambda)}{\lambda^2} \leq \E^{\frac{|t|}{2\lambda_0(\sigma_0)}} \int_\R \frac{d\rho_0(\lambda)}{\lambda^2},
 \end{align*}
 where $\rho_{k,0}$ and $\rho_0$ are the spectral measures corresponding to the pairs $(u_{k,0},\mu_{k,0})$ and $(u_0,\mu_0)$ respectively. 
 Similarly, the definition of $P_k$ gives the uniform bound  
 \begin{align*}
   \E^{-x} P_k(x,t) \leq \int_\R \E^{-s}d\mu_{k}(s,t) \leq   \E^{\frac{|t|}{2\lambda_0(\sigma_0)}} \int_\R \frac{d\rho_0(\lambda)}{\lambda^2}. 
 \end{align*}
 As the sequence $(u_{k,0},\mu_{k,0})$ has a uniform spectral gap, it belongs to $\CHdom(R)$ for some $R>0$, so that we can infer from Proposition~\ref{propFlowCont} that the pairs $(u_k(\ledot,t),\mu_k(\ledot,t))$ converge to 
 \begin{align*}
   (u(\ledot,t),\mu(\ledot,t)) = \Phi^t(u_0,\mu_0)
 \end{align*}
 in $\CHdom$ for every fixed $t\in\R$. 
 It then follows from Lemma~\ref{lem:topD} that the functions $u_k(\ledot,t)$ converge to $u(\ledot,t)$ locally uniformly and that the measures $\mu_k(\ledot,t)$ converge to $\mu(\ledot,t)$ in the sense of distributions.
 Moreover, the sequence of functions $P_k(\ledot,t)$ also converges locally uniformly to the function $P(\ledot,t)$ defined in~\eqref{eq:Pdef} by Corollary~\ref{cor:convPk}.  
 Together with the bounds established above, this allows us to pass to the limit $k\rightarrow\infty$ in~\eqref{eqnWSk1} and~\eqref{eqnWSk2}, which proves that the integral curve $t\mapsto \Phi^t(u_0,\mu_0)$ is a global conservative solution of the two-component Camassa--Holm system~\eqref{eqnOurCH}. 
\end{proof}

    \begin{remark}\label{rem:BressanUniq}
 The question about uniqueness of conservative weak solutions to the Camassa--Holm equation and its two-component generalization is a subtle one. 
 Uniqueness of conservative weak solutions to the Camassa--Holm equation has been established in~\cite{bcz15} (see also~\cite{bre16}) under the assumption that the initial data $u_0$ belongs to $H^1(\R)$ (notice that multi-peakon profiles are contained in $H^1(\R)$). 
 However, the notion of weak solution employed in~\cite{bcz15} is stronger than ours, so that this uniqueness result does not apply in our case. 
On the other hand, we have seen in the proof of Theorem~\ref{thmWeakSol} that our weak solution of the two-component Camassa--Holm system~\eqref{eqnOurCH} can be approximated by a sequence of conservative multi-peakon solutions in a certain way; compare \cite{hora06, hora08}. Moreover, the set of all multi-peakon profiles $\Peakons$ is clearly invariant under the conservative Camassa--Holm flow $\Phi$ and it has been proved in~\cite{ConservMP} that the conservative Camassa--Holm flow on $\Peakons$ gives rise to the same conservative multi-peakon solutions that had been constructed before in~\cite{brco07} and~\cite{hora07a,hora07}.  
Thus, Theorem~\ref{thmWeakSol} asserts that the conservative Camassa--Holm flow on $\Peakons$ extends continuously and uniquely to bounded subsets of $\CHdom$.
This can be viewed as a well-posedness result of the two-component Camassa--Holm system on $\CHdom$. 
  \end{remark}
      
      \begin{remark}\label{rem:Optimal}
       It is possible to extend the flow defined by~\eqref{eqnSMEvo} to a slightly larger class of measures. 
       More precisely, this flow is still well-defined on the set of all positive Poisson integrable Borel measures $\rho$ on $\R$ such that  
\begin{align}\label{eq:optimal}
\int_{(-1,1)}\E^{\frac{t}{2\lambda}}d\rho(\lambda) <\infty
\end{align}
holds for all $t\in\R$.
  Obviously, this condition requires the measures $\rho$ to decay exponentially near zero. 
 For this reason, our assumption that zero does not belong to the support of the measures $\rho$ (which is equivalent to the growth restriction in~\eqref{eqnMdef+}) appears to be rather close to optimal.
 Relaxations of condition~\eqref{eqnMdef+} corresponding to this extension of the flow would most likely not significantly expand the phase space. 
 Unfortunately, despite recent progress in the study of high and low energy asymptotics of spectral measures (see~\cite{AsymCS} for example), a complete characterization of coefficients that give rise to spectral measures satisfying~\eqref{eq:optimal} seems to be out of reach. 
 However, even with such a characterization, it is not obvious whether the time evolution~\eqref{eqnSMEvo} for measures would then still give rise to weak solutions of the two-component Camassa--Holm system~\eqref{eqnOurCH}.
\end{remark}

     \begin{remark}\label{remEACons}
     The functional $E$ defined in~\eqref{eq:NormDef} satisfies the two-sided estimate 
 \begin{align}
     \frac{1}{6\lambda_0(\sigma)} \leq E(\Phi^t(u_0,\mu_0)) \leq \frac{\sqrt{2}}{\lambda_0(\sigma)}
 \end{align}
 for all times $t\in\R$, where $\lambda_0(\sigma)$ is the size of the spectral gap around zero as defined in~\eqref{def:lam0} of the underlying spectrum $\sigma$ for the pair $(u_0,\mu_0)$.
  In particular, this allows to give a uniform two-sided estimate for $E(\Phi^t(u_0,\mu_0))$ that only depends on the corresponding quantity at initial time. 
  More precisely, one has   
  \begin{align}\label{eq:ConservE}
\frac{1}{6\sqrt{2}}E(u_0,\mu_0) \le E(\Phi^t(u_0,\mu_0))\le 6\sqrt{2} E(u_0,\mu_0)
\end{align}
 for all times $t\in\R$. 
 Since the gap around zero of the essential spectrum is conserved as well, the corresponding quantity when the supremum in~\eqref{eq:NormDef} is replaced by a limes superior at $\infty$ obeys a similar two-sided estimate; see Remark~\ref{remEssSpecGap}. 
   \end{remark}    
      
    We have seen that invariance of the spectrum allows us to control our weak solutions. 
   In the following, we are going to list a number of further properties that are preserved under the flow for the same reason.
  To this end, let us denote with $\CHdom_\infty$ the set of all pairs $(u,\mu)$ in $\CHdom$ such that
  \begin{align}
    \lim_{x\rightarrow\infty} \E^{x}\biggl(\int_{x}^{\infty}\E^{-s}(u(s) + u'(s))^2ds + \int_{x}^{\infty}\E^{-s}d\dip(s)\biggr) = 0.
  \end{align}
  Furthermore, for each $p>1$, let $\CHdom_p$ be the set of all pairs $(u,\mu)$ in $\CHdom$ such that
  \begin{align}\label{eqnSpint}
    \int_{\R} \E^{\frac{px}{2}}\biggl(\int_{x}^{\infty}\E^{-s}(u(s) + u'(s))^2ds + \int_{x}^{\infty}\E^{-s}d\dip(s)\biggr)^{\nicefrac{p}{2}} dx < \infty.
  \end{align}
 According to Proposition~\ref{prop:Sp-classes}~\ref{prop:Sp-classes-i}--\ref{prop:Sp-classes-ii}, the subsets $\CHdom_\infty$ and $\CHdom_p$ for $p>1$ can be completely characterized in terms of the corresponding spectral data. 
 More precisely, we know that a pair $(u,\mu)$ in $\CHdom$ belongs to $\CHdom_\infty$ if and only if the spectrum $\sigma$ is purely discrete and that $(u,\mu)$ belongs to $\CHdom_p$ with $p>1$ if and only if the spectrum $\sigma$ satisfies 
   \begin{align}\label{eqnSinSpCH02}
    \sum_{\lambda\in\sigma} \frac{1}{|\lambda|^p} < \infty.
  \end{align}
  Consistent with this, we will use condition~\eqref{eqnSinSpCH02} on the spectrum to define the subsets $\CHdom_p$ for all $p>0$ (let us stress however that already for $p=1$ the subset $\CHdom_p$ does not allow a characterization by means of the integral~\eqref{eqnSpint} and in fact, its characterization remains an open and nontrivial problem; see~\cite{AJPR,rowo20}). 
  As the spectrum is preserved under the conservative Camassa--Holm flow, all these subsets are invariant and moreover, the trace formulas in Proposition~\ref{prop:Sp-classes} give rise to conserved quantities. 
      
   \begin{corollary}\label{cor:CHonDp}
     The following assertions hold:
\begin{enumerate}[label=(\roman*), ref=(\roman*), leftmargin=*, widest=iii]
   \item\label{itmCHonDpi} The subset $\CHdom_\infty$ is invariant under the conservative Camassa--Holm flow. 
   \item\label{itmCHonDpii}     For each $p>0$, the subset $\CHdom_p$ is invariant under the conservative Camassa--Holm flow.
   \item\label{itmCHonDpiii} The functional 
 \begin{align}\label{eqnCQH1}
     (u,\mu) & \mapsto \int_\R d\mu = \|u\|^2_{H^1(\R)}  + \int_\R d\dip
 \end{align}
 on $\CHdom_2$ is preserved under the conservative Camassa--Holm flow.
  \item\label{itmCHonDpiv} The functional 
 \begin{align}\label{eqnCQaver}
     (u,\mu) & \mapsto \lim_{x\rightarrow\infty} \int_{-\infty}^x u(s) ds
 \end{align}
 on $\CHdom_1$ is preserved under the conservative Camassa--Holm flow. 
 \item\label{itmCHonDpv} If the initial data $(u_0,\mu_0)$ belongs to $\CHdom_1$, then the measure $\dip(\ledot,t)$ corresponding to $\Phi^t(u_0,\mu_0)$ is singular (with respect to the Lebesgue measure) for all times $t\in\R$.
  \end{enumerate}
 \end{corollary}
  
\begin{proof}
 Since the sets $\CHdom_\infty$ and $\CHdom_p$ for $p>0$ can be characterized via the spectrum, items~\ref{itmCHonDpi} and~\ref{itmCHonDpii} follow from invariance of the spectrum. 
 That the functionals in items~\ref{itmCHonDpiii} and~\ref{itmCHonDpiv} are preserved follows from their representation via the spectrum in Proposition~\ref{prop:Sp-classes}~\ref{itmS2}--\ref{itmS1}.
 Item~\ref{itmCHonDpv} follows from Proposition~\ref{prop:Sp-classes}~\ref{itmS1}.
\end{proof}  
  
  The sum in~\eqref{eqnSinSpCH02} is clearly preserved under the conservative Camassa--Holm flow. 
  Since this sum can be related to the integral in~\eqref{eqnSpint} when $p>1$, this allows one to control the functional $E_p$ defined by 
\begin{align}\label{eq:EumuSp}
E_p(u,\mu) =    \int_\R \E^{\frac{px}{2}}\biggl(\int_{x}^{\infty}\E^{-s}(u(s) + u'(s))^2ds + \int_{x}^{\infty}\E^{-s}d\dip(s)\biggr)^{\nicefrac{p}{2}} dx,
\end{align}
 which gives rise to a new family of almost conserved quantities.

\begin{corollary}\label{cor:EpConserved}
For each $p>1$, there is a positive constant $C_p>0$ such that\footnote{The constant $C_p$ can be made effective. This requires making the two-sided estimates on the Schatten--von Neumann norms of the integral operators in~\cite[Theorem~3.3]{AJPR} explicit; see Appendix~C in ArXiv version~2 (\arxiv{2106.13138v2}) of~\cite{DSpec}.}
\begin{align}\label{eq:ConservEp}
\frac{1}{C_p}E_p(u_0,\mu_0)\le E_p(\Phi^t(u_0,\mu_0))\le C_p E_p(u_0,\mu_0)
\end{align}
for all times $t\in \R$.
 \end{corollary}
   
 \begin{proof}
 The two-sided estimate in the claim follows along the same lines of reasoning as in the proof of Proposition~\ref{prop:SpecGap}. Namely, since $\sigma$ coincides with the spectrum of the generalized indefinite string $(L,\wt\omega,\wt\dip)$, one needs to use~\cite[Theorem~4.6~(iv)]{DSpec} together with two-sided estimates on the Schatten--von Neumann norms of the integral operators there, which can be found in~\cite[Proof of Theorem~3.3]{AJPR}.
 \end{proof}  
   
   Let us recall that $\CHdom^+$ was defined as the set of all pairs $(u,\mu)$ in $\CHdom$ such that $\omega$ is a positive Borel measure on $\R$ and $\dip$ vanishes identically. 
  It follows immediately from Proposition~\ref{propDefinite} that the subset $\CHdom^+$ is invariant under the conservative Camassa--Holm flow; compare \cite[Corollary~3.3]{coes98}. 
  Furthermore,  we showed in Lemma~\ref{lem:D+} that $\CHdom^+$ can be identified with a certain subset of $H^1_{\loc}(\R)$ by means of the projection $(u,\mu)\mapsto u$. 
  The corresponding flow on this subset of $H^1_{\loc}(\R)$ gives rise to weak solutions $u$ of the Camassa--Holm equation~\eqref{eqnCH} in the sense that for every test function $\varphi\in C_\cc^\infty(\R\times\R)$ one has  
 \begin{align}
  \int_\R \int_\R u(x,t) \varphi_t(x,t) + \biggl(\frac{u(x,t)^2}{2} + P(x,t) \biggr) \varphi_x(x,t) \,dx \,dt = 0,
 \end{align}
 where the function $P$ is given by 
 \begin{align}
  P(x,t) =  \frac{1}{2} \int_\R \E^{-|x-s|} u(s,t)^2 ds + \frac{1}{4} \int_\R \E^{-|x-s|} u_x(s,t)^2 ds. 
 \end{align}
 In fact, this is simply what~\eqref{eqnCHsysweak1} and~\eqref{eq:Pdef} in Definition~\ref{defGCS} reduce to when the global conservative solution belongs to $\CHdom^+$.
 We note here that global weak solutions in this sense have been obtained before in~\cite{coes98c, como00}, however, under additional spatial decay assumptions at $\infty$ and by different approaches. 
 In the more challenging indefinite case, existence of global weak solutions with various kinds of spatial asymptotics, although not covering the phase space $\CHdom$, have been established in~\cite{brco07, hora07, grhora12, grhora12b} by means of elaborate transformations to Lagrangian variables. 
   
   \begin{remark}\label{rem:conslaw+infty}
  If the initial data $(u_0,\mu_0)$ belongs to $\CHdom^+$, then the two-sided estimate in Lemma~\ref{lem:D+} shows that the $L^\infty(\R)$ norm of $u+u_x$ is an almost conserved quantity. 
  More precisely, one has 
 \begin{align}\label{eq:lawW1infty}
     \frac{1}{12\sqrt{2}}\|u_0+u'_0\|_{L^\infty(\R)} \leq \|(u+u_x)(\ledot,t)\|_{L^\infty(\R)} \leq 12\sqrt{2}\|u_0+u'_0\|_{L^\infty(\R)}
 \end{align}
 for all times $t\in\R$.
\end{remark}
   
  As mentioned already above, the subset $\CHdom^+$ is invariant under the conservative Camassa--Holm flow in view of Proposition~\ref{propDefinite}.
  According to Corollary~\ref{corsigposS1}, so is the subset $\CHdom^+_1$ defined by 
  \begin{align}\label{eqnCHdomP1}
     \CHdom^+_1 = \Bigl\{(u,\mu)\in\CHdom^+\,\Bigl|\, \lim_{x\rightarrow-\infty} \E^{-x}(u(x)+u'(x)) \text{ exists and is finite}\Bigr.\Bigr\}.
  \end{align}
   
    \begin{remark}
    If the initial data $(u_0,\mu_0)$ belongs to $\CHdom^+_1$, then the limit in~\eqref{eqnCHdomP1} is given by 
  \begin{align}
\lim_{x\to -\infty}\E^{-x}(u(x,t) + u_x(x,t)) =\int_0^\infty  \frac{d\rho(\lambda,t)}{\lambda} =\int_0^\infty \E^{-\frac{t}{2\lambda}} \frac{d\rho(\lambda,0)}{\lambda}
 \end{align}
 in view of Corollary~\ref{corsigposS1}, where $\rho(\ledot,t)$ is the corresponding spectral measure of the solution at time $t$. 
 This implies that the limit evolves according to 
 \begin{align}\begin{split}
   \frac{d}{dt} \lim_{x\to -\infty}\E^{-x}(u(x,t) + u_x(x,t)) & = -\frac{1}{2} \int_0^\infty \E^{-\frac{t}{2\lambda}} \frac{d\rho(\lambda,0)}{\lambda^2} = -\frac{1}{2} \int_0^\infty  \frac{d\rho(\lambda,t)}{\lambda^2} \\ & = - \frac{1}{2} \int_\R \E^{-x}u(x,t)^2dx - \frac{1}{2} \int_\R \E^{-x}d\mu(x,t),
 \end{split}\end{align}
 where we also used relation~\eqref{eqnLCPars02} in the last step. 
\end{remark}
   
   Next, we denote with $\CHdom^+_{1,\infty}$ the set of all pairs $(u,\mu)$ in $\CHdom^+_1$ such that
  \begin{align}
    \lim_{x\rightarrow\infty} u(x)+u'(x) = 0.
  \end{align}
  Furthermore, for each $p>\nicefrac{1}{2}$, we let $\CHdom^+_{1,p}$ be the set of all pairs $(u,\mu)$ in $\CHdom^+_1$ such that $u+u'\in L^p(\R)$.
 According to Corollary~\ref{cor:Sp-classes}, we have that $\CHdom^+_{1,\infty}=\CHdom^{+}_1\cap\CHdom_\infty$ as well as $\CHdom^+_{1,p}=\CHdom^{+}_1\cap\CHdom_p$ for $p>\nicefrac{1}{2}$, which implies that all these sets are invariant under the flow~\eqref{eqnSMEvo}.
 We will again use the latter relation to define $\CHdom^+_{1,p}=\CHdom^{+}_1\cap\CHdom_p$ for all $p>0$, without having an explicit characterization in terms of the coefficients for these sets when $p\leq\nicefrac{1}{2}$. 

   \begin{corollary}\label{cor:CHonDp+}
     The following assertions hold:
\begin{enumerate}[label=(\roman*), ref=(\roman*), leftmargin=*, widest=iii]
      \item The subset $\CHdom^{+}_{1,\infty}$ is invariant under the conservative Camassa--Holm flow. 
   \item\label{itmCHonDP+ii}     For each $p>0$, the subset $\CHdom^{+}_{1,p}$ is invariant under the conservative Camassa--Holm flow.
  \item The functional 
 \begin{align}\label{eqnCQH1+}
     (u,\mu) & \mapsto \|u\|^2_{H^1(\R)} 
 \end{align}
 on $\CHdom_{1,2}^+$ is preserved under the conservative Camassa--Holm flow.
  \item The functional 
 \begin{align}\label{eqnCQaver+}
     (u,\mu) & \mapsto  \int_{\R} d\omega
 \end{align}
 on $\CHdom_{1,1}^+$ is preserved under the conservative Camassa--Holm flow.
 \item If the initial data $(u_0,\mu_0)$ belongs to $\CHdom_{1,1/2}^+$, then the measure $\omega(\ledot,t)$ corresponding to $\Phi^t(u_0,\mu_0)$ is singular (with respect to the Lebesgue measure) for all times $t\in\R$.
   \end{enumerate}
 \end{corollary}
 
\begin{proof}
 This follows for the same reasons as in the proof of Corollary~\ref{cor:CHonDp}, one only needs to use Corollary~\ref{cor:Sp-classes} instead of Proposition~\ref{prop:Sp-classes}.
\end{proof} 
 
In conclusion, let us mention that one can improve on Corollary~\ref{cor:CHonDp+}~\ref{itmCHonDP+ii} in a similar way as in Corollary~\ref{cor:EpConserved} (compare also Remark~\ref{rem:conslaw+infty}). 

\begin{corollary}\label{cor:EpConserved+}
   If the initial data $(u_0,\mu_0)$ belongs to $\CHdom_1^+$, then for each $p>\nicefrac{1}{2}$, there is a positive constant $c_p>0$ such that
\begin{align}\label{eq:lawW1p+}
\frac{1}{c_p}\|u_0+u'\|_{L^p(\R)} \le \|(u+u_x)(\ledot,t)\|_{L^p(\R)} \le c_p\|u_0+u_0'\|_{L^p(\R)} 
\end{align}
for all times $t\in \R$.
 \end{corollary}
 
\begin{proof}
Since the function in \eqref{eqnEumon} is non-increasing and non-negative on $\R$ whenever $(u,\mu)\in \CHdom_1^+$, it is not difficult to get the estimate
\begin{align*}
 \|u+u_x\|^p_{L^p(\R)}\lesssim_p E_p(u,\mu) \lesssim_p  \|u+u_x\|^p_{L^p(\R)}
\end{align*}
for $p\in [1,\infty)$. Therefore, the corresponding claim is a consequence of Corollary~\ref{cor:EpConserved} whenever $p\in (1,\infty)$. If $p\in (\nicefrac{1}{2},1]$, then one needs to use~\cite[Theorem~5.3~(iii)]{DSpec}, which together with the facts in Appendix~C in ArXiv version~2 of~\cite{DSpec} (see the discussion of~(C.14) there), implies the desired claim.
\end{proof} 
 
\begin{remark}
  Obviously, the larger subsets $\CHdom^+\cap\CHdom_\infty$ and $\CHdom^+\cap\CHdom_p$ for $p>0$ are invariant under the conservative Camassa--Holm flow as well.
  According to Remark~\ref{remNotKreinReg}, these sets can also be described in a simpler way in terms of the coefficients as in Corollary~\ref{cor:Sp-classes} when $p\geq1$. 
  Moreover, the functional in~\eqref{eqnCQH1+} is clearly also preserved on the subset $\CHdom^+\cap\CHdom_2$ and it follows from the trace formula in~\cite[Proposition~3.3]{IsospecCH} that the functional in~\eqref{eqnCQaver+} is preserved on $\CHdom^+\cap\CHdom_1$ too.  
\end{remark}

\begin{remark}
Our phase space $\CHdom$ is not symmetric with respect to the change of variables $x\mapsto -x$, although, as remarked in the introduction, conditions~\eqref{eqnMdef-} and~\eqref{eqnMdef+} can be switched due to the symmetry 
\begin{align}
  (x,t) \mapsto (-x,-t)
\end{align}
of the two-component Camassa–Holm system. 
In the latter situation, the quantity $u+u_x$ in all the obtained (almost) conserved quantities must be replaced by $u-u_x$. 
For example, the estimates in~\eqref{eq:lawW1infty} and~\eqref{eq:lawW1p+} turn into the estimates 
\begin{align}\label{eq:lawW1p-}
\frac{1}{c_p}\|u_0-u_0'\|_{L^p(\R)} \le \|(u-u_x)(\ledot,t)\|_{L^p(\R)} \le c_p\|u_0-u_0'\|_{L^p(\R)}. 
\end{align}
Since all of these estimates, \eqref{eq:lawW1infty}, \eqref{eq:lawW1p+} and~\eqref{eq:lawW1p-}, hold true when the initial data $(u_0,\mu_0)$ has sufficient decay at both endpoints (and of course also belongs to $\CHdom^+$), it appears that the Sobolev norm $\|u\|_{W^{1,p}(\R)} = \|u\|_{L^p(\R)} + \|u_x\|_{L^p(\R)}$ also serves as an (almost) conserved quantity for all $p\in (\nicefrac{1}{2},\infty]$. 
This suggests that it is possible to extend the conservative Camassa--Holm flow to the corresponding phase spaces. 
\end{remark}

\section*{Acknowledgments}
We are grateful to the anonymous referees for several corrections and suggestions that helped to improve our manuscript as well as to Adrian Constantin, Katrin Grunert, Helge Holden and Darryl Holm for useful comments and hints with respect to the literature.

\end{document}